\documentclass[reqno, 12pt]{amsart}
\usepackage{mathrsfs}
\usepackage{amscd}
\usepackage{amsmath}
\usepackage{latexsym}
\usepackage{amsfonts}
\usepackage{amssymb}
\usepackage{amsthm}
\usepackage{graphicx}
\usepackage{hyperref}
\usepackage{makecell}
\usepackage{array}
\usepackage{booktabs}
\usepackage{multirow}
\usepackage{color,xcolor}

\newtheorem{Co}{Condition}[section]

\newtheorem{theorem}{Theorem}[section]

\newtheorem{lemma}[theorem]{Lemma}

\newtheorem{remark}[theorem]{Remark}

\newcommand\R{\mathbb{R}}

\newcommand{\be}{\begin{eqnarray}}
\newcommand{\ben}{\begin{eqnarray*}}
\newcommand{\en}{\end{eqnarray}}
\newcommand{\enn}{\end{eqnarray*}}

\newtheorem{lemm}[theorem]{Lemma}
\newtheorem{theo}[theorem]{Theorem}

\numberwithin{equation}{section}

\definecolor{rot}{rgb}{0.000,0.000,0.000}
\definecolor{blue}{rgb}{0.000,0.000,1.000}
\newcommand{\tcr}{\textcolor{rot}}

\title[Inverse source problem]{\tcr{Time-harmonic} elastic scattering by unbounded deterministic and random rough surfaces in three dimensions}

\author[G. Hu]{Guanghui Hu}\address{School of Mathematical Sciences and LPMC, Nankai University, 300071 Tianjin, China}
\email{ghhu@nankai.edu.cn}

{\author[T. Wang]{Tianjiao Wang}\address{School of Mathematical Sciences, Zhejiang University, Hangzhou 310058, China}}\email{wangtianjiao@zju.edu.cn}

\author[X. Xu]{Xiang Xu}
\address{School of Mathematical Sciences, Zhejiang University, Hangzhou 310058, China}
\email{xxu@zju.edu.cn}

\author[Y. Zhao]{Yue Zhao}
\address{School of Mathematics and Statistics, and Key Lab NAA--MOE, Central China Normal University,
	Wuhan 430079, China}
\email{zhaoyueccnu@163.com}


\begin{document}

\begin{abstract}

In this paper, we investigate well-posedness of time-harmonic scattering of elastic waves by unbounded rigid rough surfaces in three dimensions. The elastic scattering is caused by an $L^2$ function with a compact support in the $x_3$-direction, and
both deterministic and random surfaces are investigated via  the variational approach.
The rough surface in a deterministic setting is assumed to be Lipschitz and lie
within a finite distance of a flat plane, and the scattering is caused by an inhomogeneous
term in the elastic wave equation whose support lies within some finite distance of the boundary. For the deterministic case, a stability estimate of elastic scattering by rough surface is shown at an arbitrary frequency.  It is noticed that all constants in {\it a priori} bounds are bounded by explicit functions of the frequency and geometry of rough surfaces. {\color{black} Furthermore, based on this explicit dependence on the frequency together with the measurability and $\mathbb{P}$-essentially separability of the randomness, we obtain a similar bound for the solution of the scattering by random surfaces.}
\end{abstract}

\maketitle

\section{Introduction}

This paper is concerned with the mathematical analysis of the time-harmonic elastic scattering from unbounded deterministic and random rough surfaces in three dimensions. The phrase rough means surface is a (usually nonlocal) perturbation of an infinite plane  such that the whole surface
lies within a finite distance of the original plane. Rough surface scattering problems have important applications in diverse
scientific areas such as remote sensing, geophysics, outdoor sound propagation,
radar techniques (see e.g.,\cite{Abubakar,A73} and the references cited therein). In linear elasticity, the  existence and uniqueness of solution were studied in via the boundary integral equation method  \cite{Aren99, a-siaa01,Aren02}. The variational approach was proposed in \cite{eh-2010, eh-mmmas12} to handle well-posedenss of the scattering problems in periodic
structures by using the Rayleigh expansion condition (REC) and in
\cite{EH-SIMA,eh-2015} for general rigid rough surfaces by using the angular
spectrum representation (ASR).

Recently, in \cite{hlz} a mathematical formulation of the elastic rough surface scattering problems was presented in three
dimensions. Based on a Rellich-type identity, the uniqueness of weak solutions to the variational  problem was proved
if the rigid surface was the graph of a uniformly Lipschitz continuous function. The existence of solutions was also proved for the case of locally perturbed scattering
problems. However, the well-posedness problem for the scattering by a general rough surface remains unsolved.
Later, the authors in \cite{WLX} further derived an {\it a priori} bound which was explicitly dependent on the frequency.
The main goal of this paper is three-fold. First, we present a variational formulation of the elastic scattering in three dimensions by a Lipschitz-type rough surface and prove its well-posedness.
Second, we derive an {\it a priori} bound which is explicitly dependent on the frequency. {\color{black}Third, we utilize the explicit bound to derive the well-posedness for scattering by random rough surfaces as \cite{BLX,WLX}.}
As discussed in \cite{CM}, we expect that the variational formulation will be suitable for numerical solution via finite
element discretization. Furthermore, the explicit bounds we obtain should be useful in establishing
the dependence of the constants in {\it a priori} error estimates for finite element schemes on the frequency and the geometry of the domain.

This paper utilizes methods and results contained in \cite{EH-SIMA, hlz, WLX}. As was pointed out in \cite{EH-SIMA}, the elastic problem is more complicated than the acoustic case
due to the coexistence of compressional and shear waves. As a consequence, the Dirichlet-to-Neumann map for the elastic wave equation is tensor-valued which
does not have a definite real part. This brings difficulties in deriving the {\it a priori} estimates of solutions via Rellich identities for arbitrary frequencies.
We prove that the variational problem is well-posed by the theory of semi-Fredholm used in \cite{EH-SIMA}. To this end, we first consider the case of small
frequencies in which the Lax--Milgram theorem can be applied. Then we establish several {\it a priori} estimates. During this process,
we carefully trace the dependence of the coefficients of these bounds on the frequency. In this way, we arrive at an {\it a priori} bound for the solution to the variational problem
 which is explicitly dependent on the frequency. {\color{black} Afterwords, inspired by the framework for scattering by random medium in \cite{PS} and random surfaces in \cite{BLX, WLX}, we can obtain the well-posedness for  a stochastic variation problem with an explicit {\it a priori} bound.}

 The rest of this paper is outlined as follows. In Section \ref{pf} we present the variational formulation for the elastic scattering problem.
 Section \ref{small} is devoted to the well-posedness of the variational problem for small frequencies.
 In Section \ref{bound} we derive {\it a priori} bounds and trace the explicit dependence on the frequency and on the geometry of the domain.
{\color{black} For random cases, a similar bound is derived in Section \ref{random}. Conclusions are presented in Section \ref{conclusion}.}

\section{Problem formulation}\label{pf}

This section is devoted to the mathematical formulation of the three-dimensional elastic wave scattering by unbounded rigid rough surfaces. Let $D \subset \mathbb R^3$ be an unbounded connected open set such that, for some constants $m<M$,
\[
U_{M}\subset D \subset U_{m}, \quad U_h : = \{x= (x^{\prime},x_3): x_3>h\}, \quad x^{\prime}:=(x_1, x_2).
\]
The space $D$ is supposed to be filled with a homogeneous and isotropic elastic medium with unit mass density.
 We assume that $\Gamma:= \partial D$ is an unbounded rough surface, which is \tcr{supposed to be} the graph of a uniformly Lipschitz continuous function $f$. \tcr{More precisely, we assume}
\[
\Gamma = \{x\in \mathbb R^3: x_3 = f(x^{\prime}), \, x^{\prime}=(x_1, x_2) \in \mathbb R^2\},
\]
and there exists a constant $L>0$ such that
\begin{align}\label{2.1}
|f(x^{\prime}) - f(y^{\prime})|\leq L\,|x^{\prime} - y^{\prime}|  \quad \mbox{for all}\quad x^{\prime}, y^{\prime} \in \mathbb R^2.
\end{align}
Throughout the paper we fix some $h>M$. Let $\Gamma_h= \{x\in \mathbb R^3: x_3 = h\}$ and $S_h= D \backslash \overline{U}_h$. Denote the unit normal vector on $\Gamma \cup \Gamma_h$ by $\nu: = (\nu_1, \nu_2, \nu_3)$ pointing into the region of $x_3>h$ on $\Gamma_h$ and into the exterior of $D$ on $\Gamma$.
Assume that $g \in L^2(D)^3$ is an elastic source term with $\text{supp}(g)\subset S_h$. Consider the following Navier equation in three dimensions
\begin{align}\label{2.2}
\Delta^* u +
\omega^2{u}=g\quad\text{in}~ D,
\end{align}
where $\Delta^* = \mu\Delta + (\lambda+\mu)\nabla \nabla\cdot$,
$u = (u_1, u_2, u_3)^\top$ is the elastic displacement and
$\omega>0$ is the angular frequency. Here $\lambda$ and $\mu$ denote the Lam\'{e} constants characterizing the medium above $\Gamma$ satisfying $\mu>0, \lambda +2\mu/3>0$. \tcr{Since $\Gamma$ is physically rigid, there holds} the Dirichlet boundary condition
\begin{equation}\label{2.3}
u = 0\quad\text{on}~ \Gamma.
\end{equation}

As the domain $D$ is unbounded, \tcr{ a proper radiation condition should be imposed on $u$ at infinity. In this paper
we utilize} the elastic \tcr{Upward Propagation Radiation Condition} (UPRC) at infinity to ensure the well-posedness of the boundary value problem \eqref{2.2}-\eqref{2.3}. Below we  briefly introduce this radiation condition and refer to \cite{hlz,EH-SIMA} for  the details.
We begin with the decomposition of the \tcr{wave fields} into a sum of compressional and shear parts
\begin{align}\label{hd}
\tcr{u}  = \frac{1}{\rm i}(\nabla \varphi + \nabla \times  \psi), \quad \nabla \cdot  \psi = 0\quad\tcr{\mbox{in}\quad x_3>h},
\end{align}
where the scalar function $\varphi$ and the vector function $\psi$ satisfy the homogeneous Helmholtz equations
\begin{equation}\label{hlm}
\tcr{ \Delta\varphi+ k^2_{\rm p} \varphi=0, \quad \Delta  \psi+ k^2_{\rm s}  \psi= 0 \quad \text{in}\quad x_3>h.}
\end{equation}
\tcr{Here, $k_{\rm p}$ and $k_{\rm s}$ are compressional and shear wave numbers, respectively, defined by
\ben
k_{\rm p}:=\frac{\omega}{\sqrt{\lambda+2\mu}},\quad k_{\rm s}:=\frac{\omega}{\sqrt{\mu}}.
\enn
}
Denote by $\hat{v}$ the Fourier transform of $v$ in $\mathbb R^2$, i.e.,
\[
\hat{v}( \xi) =\mathcal{F}v(\xi):=\frac{1}{2\pi} \int_{\mathbb R^2} v(x^{\prime}) e^{-{\rm i}x^{\prime}\cdot  \xi}{\rm d}x^{\prime},\quad \xi=(\xi_1, \xi_2)\in \R^2.
\]
Taking the Fourier transform of \tcr{\eqref{hlm}} and assuming that $\varphi, \psi$ fulfill the \tcr{Upward Angular Spectrum Representation} (UASR) of the Helmholtz equation in $U_h$ (see \cite{CM}), we obtain
for $x_3\geq h$ that
\be\label{re}
\varphi(x^{\prime},x_3)&=\frac{1}{2\pi} \int_{\mathbb R^2} \hat{\varphi} ( \xi,h) e^{{\rm i}\beta ( \xi)(x_3-h)}
e^{{\rm i} \xi \cdot x^{\prime}}{\rm d} \xi,\notag\\
 \psi(x^{\prime},x_3)&= \frac{1}{2\pi}\int_{\mathbb R^2} \hat{ \psi} ( \xi,h) e^{{\rm i}\gamma ( \xi)(x_3-h)}
e^{{\rm i} \xi \cdot x^{\prime}}{\rm d} \xi,
\en
where
\begin{align*}
 \beta( \xi):=
 \begin{cases}
   ( k^2_{\rm p} - | \xi|^2)^{1/2},\quad
&| \xi|< k_{\rm p},\\
  {\rm i} (| \xi|^2 -  k^2_{\rm p})^{1/2},\quad
&| \xi|> k_{\rm p},
 \end{cases}
\end{align*}
and
\begin{align*}
 \gamma( \xi):=
 \begin{cases}
   ( k^2_{\rm s} - | \xi|^2)^{1/2},\quad
&| \xi|< k_{\rm s},\\
  {\rm i} (| \xi|^2 -  k^2_{\rm s})^{1/2},\quad
&| \xi|> k_{\rm s}.
 \end{cases}
\end{align*}
Denote \tcr{the Fourier transform of $\varphi(x',h)$ and $\psi(x',h)$ by }
\begin{align*}
A_{{\rm p}}(\xi) = \hat{\varphi} ( \xi,h), \quad \tilde{\boldsymbol A_{{\rm s}}}(\xi) = \hat{ \psi} ( \xi,h),
\end{align*}
\tcr{respectively.
Noting that $\mbox{div}\, \psi=0$, we have $(\xi, \gamma(\xi))\cdot  \tilde{\boldsymbol A_{\rm s}}(\xi)^{\top}=0$. For notational convenience we omit the dependence of $\beta$ and $\gamma$ on $\xi$ in the subsequent context.}

Substituting \eqref{re} into \eqref{hd}, we obtain for $x_3\geq h$ that
\begin{align}\label{uprc1}
\tcr{u}(x) =\frac{1}{2\pi} \int_{\mathbb R^2} \left[ A_{\rm p}(\xi)\,(\xi,\beta)^{\top} e^{{\rm i}\beta (x_3-h)}
 + \boldsymbol A_{{\rm s}}(\xi)\, e^{{\rm i}\gamma (x_3-h)}\right]
e^{{\rm i} \xi \cdot x^{\prime}}{\rm d} \xi,
\end{align}
where  $\boldsymbol A_{{\rm s}} = (A^{(1)}_{\rm s}, A^{(2)}_{\rm s}, A^{(3)}_{\rm s})^{\top}(\xi):=(\xi,\gamma)^{\top}\times \tilde{\boldsymbol A_{{\rm s}}}(\xi)$. It follows from \eqref{uprc1} and the orthogonality $(\xi, \gamma)\cdot \boldsymbol A_{\rm s}^{\top}=0$ that
\begin{align*}
\begin{bmatrix}
\tcr{\hat{u}}(\xi, h) \\ 0
\end{bmatrix}
 =
\begin{bmatrix}
\xi_1 & 1 & 0 & 0\\
\xi_2 & 0 & 1 & 0\\
\beta & 0 & 0 & 1 \\
0 & \xi_1 & \xi_2 & \gamma
\end{bmatrix}
\begin{bmatrix}
A_{\rm p}(\xi)\\[5pt]
\boldsymbol A_{\rm s}(\xi)
\end{bmatrix}
:=\widetilde{\mathbb{D}}(\xi)\,\boldsymbol A(\xi),
\end{align*}
which gives
\be\label{A}
\boldsymbol A(\xi)=\begin{bmatrix}
A_{\rm p}\\[5pt]
\boldsymbol A_{\rm s}
\end{bmatrix}(\xi)=\widetilde{\mathbb{D}}^{-1}(\xi)\,
\begin{bmatrix}\tcr{\hat{u}}(\xi,h) \\ 0 \end{bmatrix}
=\mathbb{D}(\xi)\, \tcr{\hat{u}}(\xi,h).
\en
Here $\mathbb{D}$ is a $4\times 3$ matrix given by
\ben
\mathbb{D}(\xi)=\frac{1}{\beta \gamma+|\xi|^2}\begin{bmatrix}
\xi_1 & \xi_2 & \gamma \\
\beta \gamma+\xi_2^2 & -\xi_1\xi_2 & -\xi_1\gamma \\
-\xi_1\xi_2 & \beta\gamma+\xi^2 & -\xi_2 \gamma \\
-\xi_1\beta & -\xi_2 \beta & |\xi|^2
\end{bmatrix}.
\enn
Using \eqref{uprc1}--\eqref{A} yields \tcr{the expression of $u$ in $U_h$}: \begin{align}\label{uprc}
 \tcr{u}(x)=\frac{1}{2\pi}\int_{\mathbb R^2}\Big\{
\frac{1}{\beta\,\gamma+| \xi|^2}  \Big(M_{\rm p}(\xi) e^{{\rm i} (\xi\cdot x^{\prime}+ \beta (x_3-h))} + M_{\rm s}(\xi) e^{{\rm i} (\xi\cdot x^{\prime}+\gamma (x_3-h))}  \Big)   \hat{ u}^{\rm sc}(\xi, h)
  \Big\} {\rm d} \xi,
\end{align}
where
\begin{align*}
M_{\rm p}(\xi)=:\begin{bmatrix}
\xi_1^2        &  \xi_1\xi_2    & \xi_1 \gamma\\
\xi_1\xi_2     & \xi_2^2        & \xi_2 \gamma\\
\xi_1\beta  & \xi_2\beta  & \beta \gamma
\end{bmatrix}
 \mbox{ and }
M_{\rm s}(\xi)=
\begin{bmatrix}
\beta\gamma+\xi_2^2 & -\xi_1\xi_2               & -\gamma \xi_1\\
-\xi_1 \xi_2             & \beta\gamma+\xi_1^2  & -\gamma \xi_2\\
-\xi_1\beta           & -\xi_2\beta            & |\xi|^2
\end{bmatrix}.
\end{align*}
The representation \eqref{uprc} \tcr{will be referred to as the upward radiation condition  for rough surface scattering problems in linear elasticity}.

Define \tcr{the surface traction operator}
\begin{align}\label{3.2}
Tu:=2\mu\partial_{\nu}u + \lambda (\nabla\cdot u)\nu + \mu \nu\times (\nabla\times u),
\end{align}
where $\nu=(\nu_1,\nu_2,\nu_3)$ stands for the normal vector on the surface. Plugging \eqref{uprc} into \eqref{3.2} yields the Dirichlet-to-Neumann (DtN) operator
on $\Gamma_h$ (cf \cite{hlz})
\be\label{DtN}
Tu = \mathcal{T} u(x'):=\frac{{\rm i}}{2\pi}
\int_{\mathbb R^2}M( \xi)
\hat{u}(\xi) e^{{\rm i}  \xi \cdot x^{\prime}}{\rm d} \xi,
\en
where $\mathcal{M}(\xi)$ is given by
\begin{align*}
\mathcal{M}(\xi)&=\frac{1}{|\xi|^2+\beta\gamma}\\
&\quad \times
\begin{bmatrix}
\mu[(\gamma - \beta)\xi_2^2 +  k_{\rm s}^2\beta] & -\mu\xi_1\xi_2(\gamma - \beta) & (2\mu| \xi|^2 - \omega^2 + 2\mu\beta \gamma)\xi_1 \\[5pt]
-\mu\xi_1\xi_2(\gamma - \beta) & \mu[(\gamma - \beta)\xi_1^2 +  k_{\rm s}^2\beta] & (2\mu| \xi|^2 - \omega^2 + 2\mu\beta\gamma)\xi_2\\[5pt]
- (2\mu| \xi|^2 - \omega^2 + 2\mu\beta \gamma)\xi_1 & -(2\mu| \xi|^2 - \omega^2 + 2\mu\beta\gamma)\xi_2 & \gamma \omega^2
\end{bmatrix}.
\end{align*}
The boundary operator $\mathcal{T}$ is non-local and is equivalent to the upward radiation condition (\ref{uprc}).
It is also called the transparent boundary condition (TBC) for time-harmonic scattering problems in unbounded domains.

Based on the above DtN operator, the wave scattering problem \eqref{2.2}-\eqref{2.3} can be reduced to \tcr{a boundary value problem over} $S_h$:
\begin{align*}
\mu\Delta{u} + (\lambda+\mu)\nabla \nabla\cdot{u} +
\omega^2{u}=g\quad&\text{in}~ S_h\\
u = 0\quad&\text{on}~ \Gamma\\
Tu = \mathcal{T} u\quad &\text{on}~ \Gamma_h.
\end{align*}

To introduce the variational formulation, we introduce the energy space $V_h$ for $h>M$ as the closure of $C_0^\infty(S_h \cup \Gamma_h)^3$ in the $H^1$ norm
\[
\|u\|_{V_h} = (\|\nabla u\|^2_{L^2(S_h)^3} + \|u\|_{L^2(S_h)^3})^{1/2}.
\]
Multiplying the Navier equation in \eqref{2.2} by the complex conjugate of the test function $v\in V_h$ and using Betti's formula yield
\ben\label{3.161}
\int_{S_h} \mathcal{E}(u, \bar{v}) - \omega^2  u \cdot  \bar{v}~{\rm d}  x - \int_{\Gamma_h}  \bar{v}\cdot T {u} {\rm d}s = \int_{S_h}g\cdot\bar{v}~{\rm d}x,
\enn
where the bilinear form $\mathcal{E}(\cdot,\cdot)$ is defined by
\ben
\mathcal{E}(u,  v) := 2\mu \sum_{j,k = 1}^3 \partial_k u_j \partial_k v_j + \lambda \nabla \cdot  u \nabla \cdot v - \mu \nabla \times  u \cdot \nabla \times  v,\quad\forall u, v\in V_h.
\enn
Define the sesquilinear  form $B: V_h\times V_h \rightarrow \mathbb{C}$ by
\begin{align}\label{3.17}
B(u,v) = \int_{S_h} \mathcal{E}(u, \bar{v}) - \omega^2  u \cdot  \bar{v}~{\rm d}  x - \int_{\Gamma_h}  \bar{v}\cdot \mathcal{T} {u} {\rm d}s.
\end{align}
Now we can formulate the variational problem as follows:

{\it Variational Problem I:}  find $u\in V_h$ such that
\begin{align}\label{3.16}
B(\tcr{u},v) = -\int_{S_h}g\cdot\bar{v}~{\rm d}x \quad \text{for all} \,\, v\in V_h.
\end{align}

The variational problem is equivalent to the boundary value problem: given $g\in L^2(D)^3$, with $\text{supp}(g)\subset S_h$ for some $h>M$, find $u\in H^1_{loc}(D)^3$ such that $u|_{S_h}\in V_h$ for every $h>M$ (implying $u=0$ on $\Gamma$), the Navier equation $(\triangle^* + \omega^2)~u = g$ in $D$ holds in a distributional sense, and the radiation condition \eqref{uprc} is satisfied with $u|_{\Gamma_h} \in H^{1/2}(\Gamma_h)^3$ by the trace theorem.

The main theorem of this paper can now be stated as follows.
\begin{theo}\label{thm}
For any $\omega>0$, the Variational Problem I (\ref{3.16}) is uniquely solvable in $V_h$. Moreover, there exists a constant $C$ independent of $\omega, h$ and the Lipschitz constant $L$ of $f$
such that the solution satisfies the estimate
\begin{align}\label{est}
\|u\|_{V_h}\leq (h-m+2)\big(C_4(\omega,h)+C_5(\omega, h)^2+C_6(\omega,h,L)\big)\;\|g\|_{V_h}
\end{align}
where
\[
C_4(\omega, h)= C(h+1-m)\omega,\quad
C_5=C\sqrt{1+\omega^{-1}}C_3(\omega,h)
    \]
    and\[
C_6=C(\omega^{-1}+1)C_1(\omega,h,L)
C_2(\omega,h,L)^2.
    \]
Here
\begin{align*}
C_1(\omega,h,L)&=C\omega^3(1+L^2)^{1/2}(h-m+1),\\
 C_2(\omega, h, L)&=C(1+L^2)^{1/4}\sqrt{h+1-m}(1+\omega (h+1-m)),\\
 C_3(\omega ,h)&=C(h+1-m)(1+\omega (h+1-m))^2/\omega.
\end{align*}
\end{theo}
The constants $C_1$-$C_6$ are derived from {\it a priori} bounds of the variational solution,
which exhibit explicit dependence on the frequency $\omega$ and the geometry of the rough surface. They lead to the explicit {\it a priori} bound of the solution of the elastic scattering problem in three dimensions.

By the semi-Fredholm theory in \cite{EH-SIMA},
the results of Theorem \ref{thm} follow from the well-posedness of the variational problem at small frequencies (cf Theorem \ref{small})
and an {\it a priori} bound of the solution to the variational problem at an arbitrary frequency (cf Theorem \ref{lemma5}).
Thus, in the subsequent two sections we shall focus on mathematical analysis at small frequencies and {\it a priori} estimate at an arbitrary frequency.

\section{Analysis of the variational problem for small frequency}

We first investigate mapping properties the DtN operator in three dimensions.
For a matrix $\mathcal{M}(\xi)\in \mathbb{C}^{3\times 3}$ depending on $\xi$, let ${\rm Re}\mathcal{M}(\xi) := (\mathcal{M}(\xi) + \mathcal{M}(\xi)^*)/2.$ We shall write ${\rm Re}\mathcal{M}(\xi)>0$ if ${\rm Re}\mathcal{M}(\xi)$ is positive definite. Here $M^*(\xi)$ is the adjoint of $M$ with respect to the scalar product $(\cdot,\cdot)_{\mathbb{C}^{3\times 3}}$ in $\mathbb{C}^{3\times 3}$.
\begin{lemm}\label{dtn}
Let $\mathcal{M}(\xi)$ be defined in \eqref{DtN} and let $h>M$.
\begin{enumerate}
\item There exists a constant $K$ independent of $\omega$ such that ${\rm Re}(-{\rm i}M)(\xi)>0$ for all $|\xi|>K\omega$, where
\ben
\tcr{K=\frac{\lambda+2\mu}{\mu\sqrt{\lambda+\mu}}}>\frac{1}{\sqrt{\mu}}.
\enn
\item The DtN map $\mathcal{T}$ is a bounded operator from $H^{1/2}(\Gamma_h)^3$ to $H^{-1/2}(\Gamma_h)^3$.
\item For $ |\xi|<K\omega$ there holds
\begin{align}\label{4.1}
\|\mathcal{M}(\xi)\| \le C_K\, \omega
\end{align}
where \tcr{\[C_K=2(\lambda+4\mu)K+(\mu (\lambda+2\mu) K^2 +2 (\lambda+2\mu)/\mu)\sqrt{\frac{\lambda+\mu}{\mu(\lambda+2\mu)}}\]} is a constant independent of $\omega,\xi$ and the norm
\[
\|\mathcal{M}(\xi)\|=\max_{1\leq  i, j\leq 3} |M_{ij}(\xi)|.
\]
Here $K$ is the constant specified in item (1)
and $M_{ij}(\xi), 1\leq  i, j\leq 3$ denote the entries of $\mathcal{M}(\xi)$.
\end{enumerate}
\end{lemm}
\begin{remark}
In comparison with properties of the matrix $M$ in two dimensions, we provide explicit constants $K$ and $C_K$ in terms of the Lame coefficients.
\end{remark}

\begin{proof}

Item (2) has been proved in \cite[Lemma 3.2]{hlz}.
Thus we only need to prove items (1) and (3).

(1) Since $|\xi|>K\omega>k_s$, we have $\beta={\rm i}|\beta|$ and $\gamma={\rm i}|\gamma|$, which implies \begin{align} \label{Le3.1.1}
    i\mathcal{M}(\xi)=\frac{-1}{|\xi|^2-|\beta||\gamma|} \left[\begin{array}{ccc}
        a_1(\xi) & b(\xi) & -ic(\xi)\xi_1 \\ b(\xi)
         & a_2(\xi) & -ic(\xi)\xi_2 \\
        ic(\xi)\xi_1 & ic(\xi)\xi_2 & a_3(\xi)
    \end{array}\right] := \frac{-1}{|\xi|^2-|\beta||\gamma|} \mathcal{M}'(\xi)
\end{align} with \begin{align*}
    a_1(\xi)&=\mu [\xi_2^2(|\gamma|-|\beta|)+k_s^2|\beta|],\quad a_2(\xi)=\mu [\xi_1^2(|\gamma|-|\beta|)+k_s^2|\beta|], \quad a_3(\xi)=\omega^2|\gamma|, \\
     b(\xi) &=-\mu \xi_1 \xi_2 (|\gamma|-|\beta|),\quad c(\xi)=2\mu|\xi|^2-\omega^2+2\mu|\beta||\gamma|.
\end{align*} 
It is obvious that $a_i(\xi)$, $b(\xi)$, $c(\xi)$ $\in \mathbb{R}$. Then from \eqref{Le3.1.1} we obtain \begin{align*}
    \Re (-{\rm i}\mathcal{M}(\xi))= \frac{1}{|\rho|}\mathcal{M}'(\xi)
\end{align*} with $\rho(\xi)=|\xi|^2+\beta\gamma$. Hence it remains to prove $\mathcal{M}'(\xi)$ is positive-definite when $|\xi|>K \omega$. To this end, we should verify \[
\text{\rm i})\, a_1(\xi)>0,\quad \text{\rm ii})\, \left|\begin{array}{cc}
    a_1(\xi) & b(\xi) \\
   b(\xi)  & a_2(\xi)
\end{array}\right| >0 ,\quad \text{\rm iii})\, \det{\mathcal{M}'(\xi)}>0.
\]

i) By direct calculation, it is obvious that \begin{align}
a_1(\xi)&=\mu[(|\gamma|-|\beta|)\xi_2^2+k_s^2|\beta|] \notag\\ &=\mu \frac{\xi_2^2(|\gamma|^2-|\beta|^2)+k_s^2(|\beta|^2+|\beta||\gamma|)}{|\gamma|+|\beta|} \notag \\ &=\mu \frac{\xi_1k_s^2+\xi_2k_p^2+k_s^2|\beta||\gamma|-k_p^2k_s^2}{|\gamma|+|\beta|} \notag \\ &\ge \mu \frac{(|\xi|^2-k_s^2)k_p^2+k_s^2|\beta||\gamma|}{|\gamma|+|\beta|} > 0. \label{Le3.1.2}
\end{align} Here the condition $|\xi|>K\omega >k_s$ is used in the last step.

ii) Denote $g(\xi)=(|\gamma|-|\beta|)|\xi|^2+k_s^2|\beta|$. Similar as \eqref{Le3.1.2} we have $g(\xi)>0$. Then one arrives at \begin{align*}
    \left|\begin{array}{cc}
        a_1(\xi) & b(\xi) \\
        b(\xi) & a_2(\xi)
    \end{array}\right|&=a_1a_2-b^2 \\ &=\mu^2[(|\gamma|-|\beta|)|\xi_1|^2+k_s^2|\beta|][(|\gamma|-|\beta|)|\xi_2|^2+k_s^2|\beta|]-\mu^2\xi_1^2\xi_2^2(|\gamma|-|\beta|)^2\\ &=\mu^2k_s^2|\beta|g(\xi)>0.
\end{align*}

iii) Denote $h(\xi)=2\xi_1\xi_2 b(\xi)-a_1(\xi)\xi_2^2-a_2(\xi)\xi_1^2$, then it can be verified that \begin{align}
    \det(\mathcal{M}'(\xi)) &=a_3(\xi)\left|\begin{array}{cc}
        a_1(\xi) & b(\xi) \\
        b(\xi) & a_2(\xi)
    \end{array}\right|+(-a_1(\xi)c(\xi)^2\xi_2^2+2b(\xi)c(\xi)^2\xi_1\xi_2-a_2(\xi)c(\xi)^2\xi_1^2) \notag \\ &= \mu^2k_s^2|\beta||\gamma|g(\xi)\omega^2+c(\xi)^2h(\xi). \label{Le3.1.3}
\end{align}
Direct calculation implies \begin{align}
    h(\xi) &=-2\mu \xi_1^2\xi_2^2(|\gamma|-|\beta|)-\mu\xi_1^2[\xi_1^2(|\gamma|-|\beta|)+k_s^2|\beta|] \notag  \\ &\quad-\mu\xi_2^2[\xi_2^2(|\gamma|-|\beta|)+k_s^2|\beta|] \notag \\ &=-\mu(|\gamma|-|\beta|)|\xi|^4-\mu k_s^2 |\beta| |\xi|^2= -\mu |\xi|^2 g(\xi).\label{Le3.1.4}
\end{align} Combining \eqref{Le3.1.3}-\eqref{Le3.1.4} gives \begin{align*}
    \det(\mathcal{M}'(\xi)) &=\mu^3 g(\xi) \{k_s^4|\beta||\gamma|-|\xi|^2[2|\gamma|(|\gamma|-|\beta|)+k_s^2]\} \\ &= \mu^3 g(\xi) \left\{k_s^4|\beta||\gamma|-\left[\frac{2|\gamma|(k_p^2-k_s^2)}{|\beta|+|\gamma|}+k_s^2\right]^2\right\} \\ &= \mu^3 g(\xi) \frac{d(\xi)}{(|\gamma|+|\beta|)^2}
\end{align*} with \[
d(\xi)= k_s^4(|\gamma||\beta|-|\xi|^2)(|\gamma|+|\beta|)^2+4|\gamma|(k_s^2-k_p^2)(|\gamma|k_p^2+|\beta|k_s^2)|\xi|^2.
\] Hence we only need to verify $d(\xi)>0$ for $|\xi|>K \omega$. Taking $|\xi|^2=K'k_s^2$ implies \begin{align*}
    d(\xi)&=k_s^8[(\sqrt{(K'-\alpha)(K'-1)}-K')(\sqrt{K'-1}+\sqrt{K'-\alpha})^2 \\ &+4(1-\alpha)K'\sqrt{K'-1}(\sqrt{K'-\alpha}+\alpha\sqrt{K'-1})] \\ &> k_s^8[-(\sqrt{K'-\alpha}+\sqrt{K'-1})^2+4(1-\alpha)K'\sqrt{K'-1}\alpha(\sqrt{K'-\alpha}+\sqrt{K'-1})]
\end{align*} with $\alpha:= {k_p^2}/{k_s^2}={\mu}/{(\lambda+2\mu)}<1$. In order to show $d(\xi)>0$, we will verify \begin{align*}
    2(1-\alpha)\alpha K' \sqrt{K'-1}>\sqrt{K'-\alpha},
\end{align*} i.e. \begin{align}\label{Le3.1.5}
    K'\sqrt{\frac{K'-1}{K'-\alpha}}>\frac{1}{2(1-\alpha)\alpha}.
\end{align} To guarantee \eqref{Le3.1.5}, let \[
\sqrt{\frac{K'-1}{K'-\alpha}} >\frac{1}{2},\quad C >\frac{1}{\alpha(1-\alpha)},
\] i.e. \[
K'>\max\left\{\frac{4-\alpha}{3}, \frac{1}{\alpha(1-\alpha)}\right\}=\frac{1}{\alpha(1-\alpha)}=\frac{(\lambda+2\mu)^2}{\mu(\lambda+\mu)}.
\] Hence, supposing that \[
|\xi|>\sqrt{\frac{K'}{\mu}}\omega=\frac{\lambda+2\mu}{\mu\sqrt{\lambda+\mu}}\omega
\] guarantees $d(\xi)>0$, which implies $\det{\mathcal{M}'(\xi)}>0$.

    (3) For $\rho(\xi)=|\xi|^2+\beta\gamma$,
    direct calculation gives \begin{align} \label{Le3.1.6}
      \left\{\begin{array}{lll}
       k_p^2 &\le |\rho| \le k_pk_s, \quad 0\le |\xi| \le k_p; \\  k_p^2 &\le |\rho| \le k_s^2, \quad k_p\le |\xi| \le k_s; \\  c_K\omega^2  &\le |\rho| \le k^2_s, \quad k_s \le |\xi| \le K\omega,\end{array}\right.
    \end{align} with \[
    c_K =K^2-\sqrt{(K^2-1/\mu)(K^2-1/(\lambda+2\mu))} > {1/(\lambda+2\mu)}.
    \]
    Here to derive the inequality for $k_s \le |\xi| \le K \omega$ we have used the fact that the function
    \[
    \rho(\xi) = |\xi|^2 - \sqrt{k_p^2 - |\xi|^2}  \sqrt{k_s^2 - |\xi|^2}
    \]
    is decreasing with respect to $|\xi|$ for $|\xi|\geq k_s$. We also consider $\gamma-\beta$ which is
    \begin{align*}
    \gamma-\beta=\sqrt{k_s^2-|\xi|^2}-\sqrt{k^2_p-|\xi|^2}=\left\{\begin{array}{ccc}
       |\gamma|-|\beta|,  & 0< |\xi| \le k_p, \\
       |\gamma|-i|\beta|,  & k_p< |\xi| \le k_s, \\
     i( |\gamma|-|\beta|),  & |\xi|>k_s.
    \end{array}\right.
    \end{align*} Then we immediately obtain
    \begin{align} \label{Le3.1.7}
    	\left\{\begin{array}{ll}
    		|\gamma-\beta| \le \sqrt{k^2_s-k^2_p} , & 0<|\xi| \le k_p \mbox{ or } |\xi|>k_s,\\
    	 |\gamma-\beta|=\sqrt{|\gamma|^2+|\beta|^2}=\sqrt{k^2_s-k^2_p}  , & k_p < |\xi| \le k_s.
    	\end{array}\right.
    \end{align}

    (3) To prove the third result, it suffices to verify the inequality
    $M_{ij} \le C \omega$, for $i,j = 1,2,3$ and $ |\xi| \le K \omega$. For $M_{33}$, by \eqref{Le3.1.6} we have \begin{align}\label{Le3.1.8}
        |M_{33}| = \left|\frac{\gamma\omega^2}{\rho}\right| \leq \left\{\begin{array}{ccc}
            \omega^2{k_s}/{k_p^2}={\omega(\lambda+2\mu)}/{\sqrt{\mu}}, &  0\le |\xi| \le k_p,\\
            \omega^2{\sqrt{k_s^2-k_p^2}}/{k_p^2}=\omega\sqrt{{(\lambda+\mu)(\lambda+2\mu)}/{\mu}},  & k_p\le |\xi| \le k_s, \\
            \omega \sqrt{K^2\omega^2-k_s^2}/c_K \omega^2= \omega  \sqrt{K^2-1/\mu}/c_K, & k_s\le |\xi| \le K \omega.
        \end{array}\right.
    \end{align}
Similarly,  $M_{23}$ and $M_{32}$ can be estimated using \eqref{Le3.1.6} by 
 \begin{align}
        |M_{23}| &=|M_{32}|  \notag \\ &= \left| \frac{2\mu \rho \xi_2-\omega^2\xi_2}{\rho}\right| \le \left\{\begin{array}{ccc}
            2\mu k_p + \omega^2/k_P=\omega(2\mu/\sqrt{\lambda+2\mu}+\sqrt{\lambda+2\mu}), &  0\le |\xi| \le k_p,\\
            2\mu k_s+\omega k_s/k_p^2=\omega(2\sqrt{\mu}+(\lambda+2\mu)/\sqrt{\mu}),  & k_p\le |\xi| \le k_s, \\
            2\mu K \omega + K \omega^3/c_K \omega^2=\omega(2\mu K+K/c_K), & k_s\le |\xi| \le K \omega.
        \end{array}\right. \label{Le3.1.9}
    \end{align} It is obvious that $|M_{13}|= |M_{31}| $ can also be estimated by the right-hand side of \eqref{Le3.1.9}.
  It remains to estimate $M_{11}$, $M_{22}$, $M_{12}$ and $M_{21}$. For convenience, denote \[\sqrt{k_s^2-k_p^2}=\omega \sqrt{\frac{\lambda+\mu}{\mu(\lambda+2\mu)}}:=  C_{\lambda,\mu} \omega. \]
   Combining \eqref{Le3.1.6}-\eqref{Le3.1.7} gives \begin{align}
       |M_{11}| &\le \mu \left|\frac{(\gamma-\beta)\xi_2^2+k_s^2\beta}{\rho}\right| \notag \\ &\le \left\{\begin{array}{ccc}
           \mu(\omega C_{\lambda,\mu} + k_s^2/k_p) =\omega(\mu C_{\lambda,\mu}+\sqrt{\lambda+2\mu}) , &  0\le |\xi| \le k_p,\\
            \mu (C_{\lambda,\mu}\omega k^2_s/k^2_p+\omega C_{\lambda,\mu} k^2_s/k_p^2) = 2\omega C_{\lambda,\mu}(\lambda+2\mu)/\mu,  & k_p\le |\xi| \le k_s, \\
             \omega(\mu C_{\lambda, \mu} K^2/c_K +\sqrt{K^2-1/(\lambda+2\mu)}/ c_K), & k_s\le |\xi| \le K \omega.
        \end{array}\right. \label{Le 3.1.10}
   \end{align}  Obviously, $|M_{22}|$ can also be estimated by the right-hand side of \eqref{Le 3.1.10}.
For $|M_{12}|$ and $|M_{21}|$, we combine \eqref{Le3.1.6}-\eqref{Le3.1.7} to obtain \begin{align}
    |M_{12}| &=|M_{21}| \notag \\ &\le \mu \left|\frac{\xi_1\xi_2(\gamma-\beta)}{\rho}\right| \le \left\{ \begin{array}{ccc}
        \omega \mu C_{\lambda,\mu}, & 0\le |\xi| \le k_p, \\
        \mu k_s^2 \omega C_{\lambda,\mu}/k_p^2 =\omega (\lambda+2\mu)C_{\lambda,\mu}, & k_p\le |\xi| \le k_s,\\ \mu K^2 C_{\lambda,\mu}\omega^3/c_K \omega^2 = \omega(\mu K^2 C_{\lambda,\mu})/c_K, & k_s\le |\xi| \le K \omega.
    \end{array}\right. \label{Le3.1.11}
\end{align} Combining the above results \eqref{Le3.1.8}-\eqref{Le3.1.11}, we have \begin{align*}
    \|M\| \le \left\{\begin{array}{ccc}
         C_{K,1}\omega, & 0\le |\xi| \le k_p, \\
        C_{K,2}\omega, & k_p\le |\xi| \le k_s, \\
        C_{K,3}\omega, & k_s\le |\xi| \le K \omega
    \end{array}\right.
\end{align*} with \begin{align*}
    C_{K,1} &=\max \left\{\frac{\lambda+2\mu}{\sqrt{\mu}},\, \frac{2\mu}{\sqrt{\lambda+2\mu}}+\sqrt{\lambda+2\mu},\, \mu C_{\lambda,\mu}+\sqrt{\lambda+2\mu},\, \mu C_{\lambda,\mu}\right\}, \\
     C_{K,2} &=\max \left\{(\lambda+2\mu)C_{\lambda,\mu},\, 2\sqrt{\mu}+\frac{\lambda+2\mu}{\sqrt{\mu}},\, \frac{2 C_{\lambda,\mu}(\lambda+2\mu)}{\mu}\right\},\\
     C_{K,3} &=\max \left\{\frac{\sqrt{K^2-\frac{1}{\mu}}}{c_K},\, 2\mu K + \frac{K}{c_K}, \,  \frac{\mu K^2 C_{\lambda,\mu}}{c_K},\, \frac{\mu K^2 C_{\lambda,\mu}}{ c_K} + \frac{\sqrt{K^2-\frac{1}{\lambda+2\mu}}}{c_K} \right\}.
\end{align*}
   It can be verified that \[
   C_{K,1} \le 2 \frac{\lambda+2\mu}{\sqrt{\mu}}, \quad C_{K,2} \le \frac{\lambda+2\mu}{\sqrt{\mu}}+2\frac{\lambda+2\mu}{\mu}C_{\lambda,\mu}+2\sqrt{\mu}
   \] and \[
   C_{K,3} \le \frac{K}{c_K}+\frac{\mu K^2 C_{\lambda,\mu}}{c_K}+2\mu K.
   \] Recalling that $c_K>1/(\lambda+2\mu)$, we have \begin{align*}
   \max\{C_{K,1},\, C_{K,2},\, C_{K,3}\} &\le 2(\lambda+4\mu)K+(\mu (\lambda+2\mu) K^2 +2 (\lambda+2\mu)/\mu) C_{\lambda,\mu} \\ &=2(\lambda+4\mu)K+(\mu (\lambda+2\mu) K^2 +2 (\lambda+2\mu)/\mu)\sqrt{\frac{\lambda+\mu}{\mu(\lambda+2\mu)}}.
   \end{align*}
   The proof is completed.
\end{proof}

\tcr{Recall that there exists a constant $C_0=C_0(h, L,m,M)>0$ independent of $\omega$ such that
\be\label{C}
\|\nabla u\|^2_{L^2(S_h)^3}\geq 1/C_0\, ||u||_{V_h}^2,\qquad
||u||^2_{H^{1/2}(\Gamma_h)}\leq C_0\,||u||_{V_h}^2,
\en
for all $u\in V_h$. The well-posedness result for small frequencies is stated below.
}
\begin{theo}\label{small}Let $K, C_K>0$ be given as in Lemma \ref{dtn}. Then there exists a
 small frequency $\omega_0>0$  such that
the variational problem admits a unique solution in $V_h$ for all $\omega\in (0,\omega_0]$. 
\end{theo}

\begin{proof} \tcr{It is clear that $\|\nabla \times u\|^2_{L^2(S_h)^3}\leq \|\nabla  u\|^2_{L^2(S_h)^3}$.}
Now it follows from the definition of $B$ and Lemma \ref{dtn} that
\begin{align}\label{4.6}
\Re B(u,u) &=  2\mu\|\nabla u\|^2_{L^2(S_h)^3} + \lambda\|\nabla \cdot u\|^2_{L^2(S_h)^3}- \mu\|\nabla\times u\|^2_{L^2(S_h)^3} \notag\\
&\quad\quad  -\omega^2\|u\|^2_{L^2(S_h)^3}- \Re\int_{\Gamma_h}  \bar{u}\cdot \mathcal{T} {u} {\rm d}s\notag\\
&= 2\mu\|\nabla u\|^2_{L^2(S_h)^3} + \lambda\|\nabla \cdot u\|^2_{L^2(S_h)^3} - \mu\|\nabla\times u\|^2_{L^2(S_h)^3}  -\omega^2\|u\|^2_{L^2(S_h)^3} \\
&\quad \quad+ \int_{|\xi|\leq  K\omega}{\rm Re}(-{\rm i}\mathcal{M}(\xi))\hat{u}\cdot\bar{\hat{u}}{\rm d}\xi  + \int_{|\xi|>  K\omega}{\rm Re}(-{\rm i}\mathcal{M}(\xi))\hat{u}\cdot\bar{\hat{u}}{\rm d}\xi\notag\\
&\geq \mu\|\nabla u\|^2_{L^2(S_h)^3} - \omega^2\|u\|^2_{L^2(S_h)^3} + \int_{|\xi|\leq  K\omega}{\rm Re}(-{\rm i}\mathcal{M}(\xi))\hat{u}\cdot\bar{\hat{u}}{\rm d}\xi\notag\\
&\geq \mu\|\nabla u\|^2_{L^2(S_h)^3} - \omega^2\|u\|^2_{L^2(S_h)^3} - \tcr{C_K\,C_0}\omega\|u\|^2_{V_h},
\end{align}
where the constant \tcr{$C_K>0$ is given by Lemma \ref{dtn} (3)} and the constant $C_0$ is specified in \eqref{C}.
By Lemma 3.4 in \cite{CM} we have the following Poincare's inequality
\begin{align}\label{4.7}
\|u\|^2_{L^2(S_h)^3}\leq \textcolor{black}{\tcr{(h-m)}}\|\partial_3u\|^2_{L^2(S_h)^3} \leq \textcolor{black}{\tcr{(h-m)}}\|\nabla u\|^2_{L^2(S_h)^3}, \quad u\in V_h.
\end{align}
Using \eqref{4.6}-\eqref{4.7}, we obtain the estimate
\ben
\Re B(u,u)\geq \big(\mu/C_0-\omega C_0\, C_K-\omega^2 (h-m)\big)\,\|u\|^2_{V_h}\\
\geq \tcr{\big(\mu/C_0-\omega_0 C_0\, C_K-\omega_0^2 (h-m)\big)\,\|u\|^2_{V_h}}\enn
for all $u\in V_h$ and $\omega\in (0,\omega_0]$.
Choose $\omega_0$ sufficiently small such that
\ben
\mu/C_0-\omega_0 C_0\, C_K-\omega_0^2 (h-m)>0.
\enn
The proof is completed by applying the Lax-Milgram theorem.
\end{proof}

\section{An {\it a priori} bound for smooth rough surfaces}\label{bound}

In this section, we establish an {\it a priori} bound for a smooth rough surface at any frequency. The attractive feature is that all
constants in the {\it a priori} estimates are bounded
by explicit functions of $\omega$, $h, m, M$ and $L$.

\begin{lemm}\label{lemma6}
Let $u\in V_h$ be a variational solution to \eqref{3.16} with $g\in V_h$. We have
\begin{align*} \textcolor{black}{
\|\nabla\cdot u\|^2_{L^2(\Gamma)},\; \|\nabla\times u\|^2_{L^2(\Gamma)^3}} \leq {C_1}\;\|g\|_{L^2(S_h)^3}\|\partial_3u\|_{L^2(S_h)^3},
\end{align*} where $C_1=\tcr{4\mu^{-1}{(1+L^2)^{1/2}(\omega/\sqrt{\mu}\,(h-m)+1)}}$.

\end{lemm}

\begin{proof}

By \tcr{\cite[Lemma 4.1]{hlz}(see also \cite[Lemma 5]{eh-mmmas12} for the periodic version)} we have the following Rellich identity
\begin{align}\label{rellich}
&2\Re \int_{S_h} (\mu \Delta u + (\lambda + \mu)\textcolor{black}{\nabla}\nabla \cdot u + \omega^2 u)\cdot \partial_3 \bar{u}{\rm d}x\notag\\
&= \Big(-\int_{\Gamma} + \int_{\Gamma_h}\Big) \Big\{2\Re(Tu \cdot \partial_3 \bar{u}) - \nu_3 \mathcal{E}(u,\bar{u}) + \omega^2|u|^2 \Big\}{\rm d}s,
\end{align}
and
\begin{align}\label{identity}
 Tu \cdot \partial_3 \bar{u} = \textcolor{black}{\nu_3}\mathcal{E}(u,\bar{u})= \mu|\partial_{\nu}u|^2\nu_3 + \nu_3(\lambda + \mu) |\nabla \cdot u|^2.
\end{align}
From \cite[Lemma 4.2 (ii)]{hlz} we also have the following two identities
\begin{align}\nonumber
& \int_{\Gamma_h} \Big\{2\Re(Tu \cdot \partial_3 \bar{u}) -  \mathcal{E}(u,\bar{u}) + \omega^2|u|^2 \Big\}{\rm d}s\\ \label{eq:4}
=& \;  2\omega^2  \int_{|\xi|< k_{\rm p}} \beta^2(\xi) |A_{\rm p}(\xi)|^2\,{\rm d}\xi+ 2\mu
\int_{|\xi|< k_{\rm s}} \gamma^2(\xi) |\textbf{A}_{\rm s}(\xi)|^2\,{\rm d}\xi\\
=& \tcr{\;  2\omega^2 \Big\{ \int_{|\xi|< k_{\rm p}} \beta^2(\xi) |A_{\rm p}(\xi)|^2\,{\rm d}\xi+
\int_{|\xi|< k_{\rm s}} \gamma^2(\xi) |\tilde{\textbf{A}}_{\rm s}(\xi)|^2\,{\rm d}\xi
\Big\}},\nonumber\\
 \textcolor{black}{\Im}\int_{\Gamma_h} Tu \cdot \bar{u}{\rm d}s &=
\int_{|\xi|< k_{\rm p}} \omega^2\beta( \xi)|A_{\rm p}(\xi)|^2 {\rm d} \xi + \int_{|\xi|< k_{\rm s}}\mu\gamma(\xi)|\boldsymbol A_{\rm s}( \xi)|^2{\rm d} \xi \nonumber\\
&\tcr{=\omega^2\Big\{\int_{|\xi|< k_{\rm p}} \beta( \xi)|A_{\rm p}(\xi)|^2 {\rm d} \xi + \int_{|\xi|< k_{\rm s}}\gamma(\xi)|\tilde{\boldsymbol{A}}_{\rm s}( \xi)|^2{\rm d} \xi\Big\}}\label{eq:a}.
\end{align}
\tcr{Here we have used the relation $|\boldsymbol A_{\rm s}( \xi)|^2=k_{\rm s}^2
|\tilde{\boldsymbol {A}}_{\rm s}( \xi)|^2$. Note that the identity \eqref{eq:4} corrects a mistake made in \cite[Formula (4.1)]{hlz}.}
Hence, combing \eqref{rellich} and \eqref{identity} gives
\begin{align}\label{5.11}
&-\int_{\Gamma} \mu|\partial_{\nu}u|^2\nu_3 + \nu_3(\lambda + \mu) |\nabla \cdot u|^2~{\rm d}s\notag\\
\quad = & \int_{\Gamma_h} 2\Re(Tu \cdot \partial_3 \bar{u}) - \nu_3 \mathcal{E}(u,\bar{u}) + \omega^2|u|^2~{\rm d}s - 2\Re\int_{S_h} g\partial_3\bar{u}~{\rm d}x.
\end{align}
Using \eqref{eq:4} and \eqref{eq:a} and taking the imaginary part of \eqref{3.16}, we get
\begin{align}\label{5.14}
\int_{\Gamma_h}\Big\{ 2\Re(Tu \cdot \partial_3 \bar{u}) -  \mathcal{E}(u,\bar{u}) + \omega^2|u|^2 \Big\}{\rm d}s&\leq 2\textcolor{black}{ k_{\rm s}}\Im \int_{\Gamma_h}Tu \cdot \bar{u}{\rm d}s\notag\\
&\leq  2\textcolor{black}{ k_{\rm s}}\Im \int_{S_h} g\cdot \bar{u}{\rm d}s.
\end{align}
Combing \eqref{5.11} and \eqref{5.14} then gives the estimates
\begin{align}\label{5.15}
&-\int_{\Gamma} \mu|\partial_{\nu}u|^2\nu_3 + \nu_3(\lambda + \mu) |\nabla \cdot u|^2~{\rm d}s\notag\\
&\quad \leq 2\textcolor{black}{ k_{\rm s}} \Im\int_{S_h} g\cdot \bar{u}{\rm d}x -  2\Re\int_{S_h} g\cdot \partial_3\bar{u}{\rm d}x\notag\\
&\quad \leq 2 \textcolor{black}{\|g\|^2_{L^2(S_h)^3}\Big(\frac{\omega}{\sqrt{\mu}}\|u\|^2_{L^2(S_h)^3} + \|\partial_3u\|^2_{L^2(S_h)^3}\Big)}\notag\\
&\quad \leq \tcr{2 (\omega/\sqrt{\mu}\, (h-m)+1)}\|g\|^2_{L^2(S_h)^3}\|\partial_3u\|^2_{L^2(S_h)^3},
\end{align}
where the last identity follows from \eqref{4.7}. Since
\begin{align}\label{nu3}
\nu_3(x)=-\frac{1}{\sqrt{1+|\nabla_{x^{\prime}}f|^2}}<-(1+L^2)^{-1/2}<0
\quad\mbox{on}\,\,\Gamma,
\end{align}
from \eqref{5.15} we obtain that
\begin{align}\label{5.17}
&\|\nabla\cdot u\|^2_{L^2(\Gamma)} + \|\partial_{\nu} u\|^2_{L^2(\Gamma)^3}\notag\\
&\leq \tcr{2\mu^{-1}{(1+L^2)^{1/2}(\omega/\sqrt{\mu}\,(h-m)+1)}}\|g\|^2_{L^2(S_h)^3}\|\partial_3u\|^2_{L^2(S_h)^3}.
\end{align}
 Finally, using $u=0$ on $\Gamma$ and the identities in \cite[(4.17)]{eh-mmmas12} we have that
\[
\nu_3 |\nabla \times u|^2 = \nu_3 (|\nabla u|^2 - |\nabla\cdot u|^2) = \nu_3 (|\partial_{\nu}u|^2 - |\nabla\cdot u|^2) \quad\mbox{on}\,\,\Gamma.
\]
Thus, $\|\nabla \times u\|_{L^2(\Gamma)^3}$ can also be bounded by the right-hand side of \eqref{5.17} \tcr{multiplied by two}.

\end{proof}

We next need to derive estimates for the $L^2$ norms of the scalar function $\nabla \cdot u$ and the vector function $\nabla\times u$
on the artificial boundary $\Gamma_H$ and the strip $S_H$ where $H=h+1$. The derivation is based on the {\it a priori} bound for the Helmholtz equation in
\cite{CM}. By \eqref{hd}, we define
\ben
\varphi:=-\frac{\rm i}{k_p^2} \nabla\cdot u,\quad
\psi:=\frac{\rm i}{k_s^2}\,\nabla\times u,\qquad\mbox{in}\quad x_3>h.
\enn
Since both $\varphi$ and $\psi$ satisfy the Helmholtz equation \eqref{hlm} and the UASR \eqref{re},  one has the following Dirichlet-to-Neumann map
on the artificial boundary $\Gamma_H$:
\begin{align}\label{5.22}
\widetilde{\mathcal{T}}w = \mathcal{F}^{-1}({\rm i}\eta\mathcal{F}w), \quad w\in H^{1/2}(\Gamma_H),
\end{align}
where $w=\varphi, \psi$, and $\eta = \beta, \gamma$, respectively. Moreover, $\widetilde{\mathcal{T}}$ is a bounded linear map of $H^{1/2}(\Gamma_H)$ to $H^{-1/2}(\Gamma_H)$
by  \cite[Lemma 2.4]{CM}. From Lemma \ref{lemma6} we can estimate the $L^2$ norm of the trace $w$ on $\Gamma$ as
\begin{align}\label{5.23}
\|w\|^2_{L^2(\Gamma)^3}\leq C_1(\omega, h, L)\|g\|_{L^2(S_H)^3}\|\partial_3u\|_{L^2(S_H)^3}.
\end{align}

The following lemma provides estimates for $w$ on $S_H$ and the trace of $w$ on $\Gamma_H$.

\begin{lemm}\label{lemma7}

Assume that $w$ satisfies the Helmholtz equation
\be\label{he}
\Delta w + k^2 w = g_0 \quad \text{in} \,\, S_H, \qquad
\widetilde{\mathcal{T}}w = \mathcal{F}^{-1}({\rm i}\sqrt{k^2-\xi^2}\mathcal{F}w)\quad \text{on}\,\,\Gamma_H
\en
where $g_0\in L^2(S_H)$.
\textcolor{black} {Then there holds the estimate
\begin{align}\label{5.24}
\|w\|_{L^2(\Gamma_H)^3}\leq \|w\|_{L^2(S_H)^3}\leq \widetilde{C}_2(L,k,h)\|w\|_{L^2(\Gamma)^3} + \widetilde{C}_3(k,h)\|g_0\|_{L^2(S_H)^3}
\end{align} with \[
\widetilde{C}_2(L,k,h)=C(1+L^2)^{1/4}\sqrt{H-m}(1+k(H-m))
    \]and\[
 \widetilde{C}_3(k,h)=C(H-m)(1+k(H-m))^2/k
    .\]}
\end{lemm}

\begin{proof}

Consider the boundary value problem of finding $v\in H^1(S_H)$ such that
\begin{align}\label{5.25}
(\triangle+  k^2)v = \bar{w} \quad {\rm in}\quad S_H, \quad v = 0 \quad {\rm on}\quad \Gamma, \quad \partial_3v = \widetilde{T}v \quad {\rm on} \quad \Gamma_H.
\end{align}
By \cite[Lemma 4.6]{CM} the boundary value problem \eqref{5.25} is well-posed with the following estimate
\begin{align}\label{5.26}
 \textcolor{black}{\|\nabla v\|_{L^2(S_H)}+k\|v\|_{L^2(S_H)} \le C (1+k(H-m))^2(H-m) \|w\|_{L^2(S_H)}}.
\end{align}

We first prove that $\|\partial_\nu v\|^2_{L^2(\Gamma)^3}\leq C\|w\|^2_{L^2(S_H)^3}$ for some constant $C>0$  depending explicitly on $\omega, H$ and the Lipschitz constant $L$ of $\Gamma$. The Rellich identity for the Helmholtz equation gives:
\begin{align}\label{5.27}
&2\Re\int_{S_H}\partial_3\bar{v}(\textcolor{black}{\Delta} v +  k^2v){\rm d}x\notag\\
\quad = & \; \Big(\int_\Gamma + \int_{\Gamma_H}\Big)\{2\Re(\partial_\nu v \partial_3\bar{v}) - \nu_3|\nabla v|^2 + \nu_3 k^2|v|^2 \}{\rm d}s,
\end{align}
which can be proved in the same way as \eqref{rellich}. From the proof of \cite[Lemma 4.6]{CM}
it holds that
\begin{align}\label{5.28}
&\int_{\Gamma_H}\{2\Re(\partial_\nu v\partial_3\bar{v} - \nu_3|\nabla v|^2 + \nu_3 k^2|v|^2)\}{\rm d}s\leq 2 k\Im\int_{\Gamma_H}\bar{v}\widetilde{T}v{\rm d}s\notag\\
\quad &\; \leq 2 k\Im\int_{S_H}\bar{v}\bar{w}{\rm d}x.
\end{align}
Moreover, using the identities in (4.17) of \cite{eh-mmmas12} on $\Gamma$ and the bound for $\nu_3$ in \eqref{nu3} one has
\begin{align}\label{5.29}
&-\int_{\Gamma_H}\{2\Re(\partial_\nu v\partial_3\bar{v} - \nu_3|\nabla v|^2 + \nu_3 k^2|v|^2)\}{\rm d}s= -\int_{\Gamma}\nu_3|\partial_\nu v|^2{\rm d}s\notag\\
&\geq \textcolor{black}{(1+L^2)^{-1/2}}\|\partial_\nu v\|^2_{L^2(\Gamma)}.
\end{align}
Plugging \eqref{5.28} and \eqref{5.29} into \eqref{5.27} and using \eqref{5.26} yield the estimate
\begin{align}\label{5.30}
\|\partial_\nu v\|^2_{L^2(\Gamma)} &\leq \textcolor{black}{(1+L^2)^{1/2}} \Big\{-2\Re\int_{S_H}\bar{w}\partial_3v{\rm d}x + 2 k\Im\int_{S_H}\bar{w}\bar{v}{\rm d}x \Big\}\notag\\
&\leq \textcolor{black}{2(1+L^2)^{1/2} \|w\|_{L^2(S_H)}(k\|v\|_{L^2(S_H)}+\|\nabla v\|_{L^2(S_H)})}\notag \\ &\textcolor{black}{\le C(1+L^2)^{1/2}(H-m)(1+k(H-m))^2\|w\|^2_{L^2(S_H)} },
\end{align}
where the constants $C$ is independent of $w$.

Now we prove the second inequality in \eqref{5.24}. Following the approach of \cite[Lemma 7]{EH-SIMA}, we obtain that
\begin{align*}
\int_{S_H}\{w\textcolor{black}{\Delta} v - v \textcolor{black}{\Delta}w\}{\rm d}x &= \int_{\Gamma_H}\{w\partial_\nu v - v\partial_\nu w\}{\rm d}s + \int_{\Gamma}w\partial_\nu{\rm d}s\\
&= \int_{\Gamma_H}\{w\widetilde{T}v - v\widetilde{T}w\}{\rm d}s + \int_\Gamma w\partial_\nu v{\rm d}s\\
&= \int_\Gamma w\partial_\nu v{\rm d}s.
\end{align*}
Note that $v=0$ on $\Gamma$, and the Dirichlet-to-Neumann operator $\widetilde{T}$ defined in \eqref{5.22} is symmetric (see Lemma 3.2 in \cite{CM}). Thus,
\begin{align*}
\int_{S_H}|w|^2{\rm d}x &= \int_{S_H}w(\textcolor{black}{\Delta} v +  k^2v){\rm d}x\\
&=\int_{S_H}v(\nabla w +  k^2w){\rm d}x + \int_\Gamma w\partial_\nu v{\rm d}s\\
&= \int_{S_H}vg{\rm d}x + \int_\Gamma w\partial_\nu v{\rm d}s.
\end{align*}
Noting \eqref{5.26} and \eqref{5.30} one has
\begin{align*}
\|w\|^2_{L^2(S_H)}&\leq \|v\|^2_{L^2(S_H)}\|g\|^2_{L^2(S_H)} + \|w\|^2_{L^2(\Gamma)}\|\partial_\nu v\|^2_{L^2(\Gamma)}\\
& \le  C\sqrt{H-m}(1+L^2)^{1/4} (1+k(H-m))\|w\|_{L^2(S_H)}\|w\|_{L^2(\Gamma)}  \notag \\ & \quad + C(H-m)\frac{(1+k(H-m))^2}{k}\|w\|_{L^2(S_H)}\|g_0\|_{L^2(S_H)}.
\end{align*}
Then the following inequality is proved
\begin{align}\label{5.31}
\|w\|_{L^2(
        S_H)} \le \widetilde{C}_2(L,k,h) \|w\|_{L^2(\Gamma)}+ \widetilde{C}_3(k,h) \|g_0\|_{L^2(S_H)}.
\end{align}
To estimate the first inequality in \eqref{5.24} we use
\[
\int_{\Gamma_H}|w|^2{\rm d}s \leq \int_{\Gamma_c}|w|^2{\rm d}s, \quad \text{for ~ all}~c\in (h, H],
\]
which follows from the proof of \cite[Lemma 2.2]{CM}. Then we have
\begin{align}\label{5.32}
(H-h)\int_{\Gamma_H}|w|^2{\rm d}x\leq\int_{S_H\backslash S_h}|w|^2{\rm d}s\leq\int_{S_H}|w|^2{\rm d}s.
\end{align}
The estimate \eqref{5.24} is proved by combing \eqref{5.31} and \eqref{5.32}.

\end{proof}

Next we prove the estimates of the $L^2$ norms of $\nabla \cdot u$ and $\nabla \times u$ on $S_H$ and $\Gamma_H$. Using Lemma \ref{lemma7} for $v = \varphi$ and $\psi$ with $g_0 = -({\rm i}/\omega^2)\nabla\cdot g$ and $({\rm i}/\omega^2)\nabla\times g$ in \eqref{he} , respectively, and \eqref{5.23}, we obtain the estimate
\begin{align}
\|\nabla\cdot u\|^2_{L^2(S_H)} &+ \|\nabla\times u\|^2_{L^2(S_H)^3} \notag \\ &\le
C_2(\omega, h, L)^2C_1(\omega, h, L)\|g\|_{V_h}\|\partial_3 u\|_{L^2(S_H)^3} + C_3(\omega, h)^2 \|g\|^2_{V_h}, \label{5.33}
\end{align} where \[
 C_2(\omega, h, L)=C(1+L^2)^{1/4}\sqrt{H-m}(1+\omega (H-m))
\]
and
\[
C_3(\omega ,h)=C(H-m)(1+\omega (H-m))^2/\omega.
\]
In a similar way, from the estimates \eqref{5.24} and \eqref{5.23} we have the bound
\begin{align}
\|\nabla\cdot u\|^2_{L^2(\Gamma_H)} &+ \|\nabla\times u\|^2_{L^2(\Gamma_H)^3} \notag \\ &\leq C_2(\omega, h, L)^2C_1(\omega, h, L)\|g\|_{V_h}\|\partial_3 u\|_{L^2(S_H)^3} + C_3(\omega, h)^2 \|g\|^2_{V_h}.\label{5.34}
\end{align}

The following theorem provides the {\it a priori} bound for the solution to {\it Variational Problem I}  dependent on the frequency and geometry of the rough surface.

\begin{theo}\label{lemma5}
Assume that $\Gamma$ is given by the graph of a Lipschitz function $f$ satisfying \eqref{2.1}, and that $u\in V_h$ is a solution to the variational problem \eqref{3.16}. Then there exists a constant $C$ independent of $\omega, h$ and the Lipschitz constant $L$ of $f$ such that the following {\it a priori} bound holds
\textcolor{black}{\[\|u\|_{V_h}\leq (h-m+2)(C_4(\omega,h)+\textcolor{black}{C_5(\omega, h)}+C_6(\omega,h,L))\|g\|_{V_h}, \]
where
 \begin{align*}
&C_4(\omega, h)= C(h+1-m)\omega,\quad
C_5=C\sqrt{1+\omega^{-1}}C_3(\omega,h),\\
&C_6=C(\omega^{-1}+1)C_1(\omega,h,L)
C_2(\omega,h,L)^2.
\end{align*}
}
\end{theo}
\begin{proof}

We first assume that $f$ is smooth.
Multiplying both sides of the Navier equation by $(x_3 - m)\partial_3\bar{u}$ and using integration by parts yields
\begin{align}\label{5.35}
&2\Re\int_{S_H}(\triangle^* + \omega^2)u\cdot(x_3 - m)\partial_3\bar{u}{\rm d}x\notag\\
&\quad = \int_{S_H}\Big\{\mathcal{E}(u,\bar{u}) - 2\Re\Big\{\sum_{j=1}^3\mathcal{E}(u,(x_3-m)e_j)\partial_3\bar{u}_j\Big\}- \omega^2|u|^2\Big\}{\rm d}x\notag\\
&\quad + \Big(\int_{\Gamma_H}+\int_{\Gamma}\Big)[-\nu_3\mathcal{E}(u,\bar{u}) + 2\Re(Tu\cdot\partial_3\bar{u})+\nu_3\omega^2|u|^2](x_3 - m){\rm d}s.
\end{align}
Letting $v = u$ in the variational formulation \eqref{3.16} gives
\begin{align}\label{5.36}
&\int_{S_H}\{\mathcal{E}(u,\bar{u}) - \omega^2|u|\}{\rm d}x - \textcolor{black}{\Re}\int_{|\xi|> K\omega}\mathcal{M}(\xi)\hat{u}(\xi,H)\cdot\bar{\hat{u}}(\xi,H){\rm d}\xi\notag\\
\quad = &\; - \textcolor{black}{\Re} \int_{S_H}g\cdot \bar{u}{\rm d}x + \textcolor{black}{\Re}\int_{|\xi|\leq  K\omega}\mathcal{M}(\xi)\hat{u}(\xi,H)\cdot\bar{\hat{u}}(\xi,H){\rm d}\xi.
\end{align}
Taking the real part and using Lemma \ref{dtn} we have
\begin{align}
&\int_{S_H}\{\mathcal{E}(u,\bar{u}) - \omega^2|u|\}{\rm d}x \notag \\
\quad \leq &\; - \textcolor{black}{\Re}\int_{S_H}g\cdot \bar{u}{\rm d}x +\textcolor{black}{\Re} \int_{|\xi|\leq  K\omega}\mathcal{M}(\xi)\hat{u}(\xi,H)\cdot\bar{\hat{u}}(\xi,H){\rm d}\xi.\label{5.37}
\end{align}
From \eqref{5.35} and using \eqref{5.37} and \eqref{identity}, we have
\begin{align}\label{5.38}
&\int_{S_H}2\Re\Big\{\sum_{j=1}^3\mathcal{E}(u,(x_3-m)e_j)\partial_3\bar{u}_j\Big\}{\rm d}x\notag\\
&\quad - \int_\Gamma (x_3 - m)\{\mu|\partial_\nu u|^2 + (\lambda + \mu)|\nabla\times u|^2\}\nu_3{\rm d}s\notag\\
& = \int_{S_H}\{\mathcal{E}(u,\bar{u}) - \omega^2|u|^2\}{\rm d}x - 2\Re\int_{S_H}(\triangle^* + \omega^2)u\cdot(x_3-m)\partial_3\bar{u}{\rm d}x\notag\\
&\quad + (H-m)\int_{\Gamma_H}\{2\Re(Tu\cdot \partial_3\bar{u}) - \mathcal{E}(u,\bar{u}) + \omega^2|u|^2\}{\rm d}s\notag\\
&\leq \int_{S_H}\{-g\cdot u + 2\Re(g\cdot\partial_3\bar{u})(x_3 - m)\}{\rm d}x + \textcolor{black}{\Re \int_{|\xi|\leq  K \omega}}\mathcal{M}(\xi)~\hat{u}_H(\xi)\cdot \bar{\hat{u}}_H(\xi){\rm d}\xi\notag\\
&\quad + (H-m)\int_{\Gamma_H}\{2\Re(Tu\cdot \partial_3\bar{u}) - \mathcal{E}(u,\bar{u}) + \omega^2|u|^2\}{\rm d}s.
\end{align}
As $\|\mathcal{M}(\xi)\|\leq C \textcolor{black}{\omega}$ for all $|\xi|< K\omega$, one has
\begin{align}\label{5.39}
\textcolor{black}{\Re}
\int_{|\xi|\leq  K\omega}\mathcal{M}(\xi)~\hat{u}_H(\xi)\cdot \bar{\hat{u}}_H(\xi){\rm d}\xi\leq C\textcolor{black}{\omega}\int_{|\xi|\leq  K\omega}|\hat{u}_H(\xi)|^2{\rm d}\xi.
\end{align}
Using \eqref{5.39} and \eqref{5.34} gives
\begin{align}\label{5.40}
& \textcolor{black}{\Re}\int_{|\xi|\leq  K\omega}\mathcal{M}(\xi)~\hat{u}_H(\xi)\cdot \bar{\hat{u}}_H(\xi){\rm d}\xi\leq C \textcolor{black}{\omega k_s^2} \int_{|\xi|\leq  K\omega}\{|A_{\rm p}(\xi)|^2 + |\boldsymbol A_{\rm s}(\xi)|^2\}{\rm d}\xi\notag\\
\leq &\; C \textcolor{black}{\omega k_s^2} (\|A_p\|^2_{L^2(\mathbb R^2)} + \|\boldsymbol A_s\|^2_{L^2(\mathbb R^2)^3})\notag\\
\leq&\; \textcolor{black}{C \omega^{-1} \left(C_2(\omega, h, L)^2C_1(\omega, h, L)\|g\|_{V_h}\|\partial_3 u\|_{L^2(S_H)^3} + C_3(\omega, h)^2 \|g\|^2_{V_h}\right)}.
\end{align}

From the estimates \eqref{5.14} and \eqref{4.7} we have the following estimate for the last term  in \eqref{5.38}:
\begin{align}\label{5.41}
\int_{\Gamma_h}\{2\Re(Tu\cdot\partial_3\bar{u} - \mathcal{E}(u,\bar{u}) + \omega^2|u|^2)\}{\rm d}s&\leq 2 k_{\rm s}\Im\int_{S_H}g\cdot\bar{u}{\rm d}x\notag\\
&\leq 2 k_{\rm s}\|g\|_{V_h}\|u\|_{L^2(S_H)^3}\notag\\
&\leq C\textcolor{black}{(H-m)k_s}\|g\|_{V_h}\|\partial_3u\|_{L^2(S_H)^3}.
\end{align}
 Combing \eqref{5.40}--\eqref{5.41} and \eqref{5.38} and noting that the second term in \eqref{5.38} is nonnegative,
we have
\begin{align}\label{5.42}
&\int_{S_H}2\Re\Big\{\sum_{j=1}^3\mathcal{E}(u,(x_3-m)e_j)\partial_3\bar{u}_j\Big\}{\rm d}x\notag\\
&\leq C(H-m)\omega\|g\|_{L^2(S_H)^3}\|\partial_3 u\|_{L^2(S_H)^3} + C (\omega^{-1}+1) \notag \\
&\quad \times \left(C_2(\omega, h, L)^2C_1(\omega, h, L)\|g\|_{V_h}\|\partial_3 u\|_{L^2(S_H)^3} + C_3(\omega, h)^2 \|g\|^2_{V_h}\right).
\end{align}

Direct calculations yield
\begin{align*}
&\mathcal{E}(u,(x_3-m)e_1)\partial_3\bar{u}_1
=2\mu |\partial_3 u_1|^2-\mu (\partial_3u_1-\partial_1u_3)\,\partial_3\overline{u}_1,\\
&\mathcal{E}(u,(x_3-m)e_2)\partial_3\bar{u}_2
=2\mu |\partial_3 u_2|^2+\mu (\partial_2u_3-\partial_3u_2)\,\partial_3\overline{u_2},\\
&\mathcal{E}(u,(x_3-m)e_3)\partial_3\bar{u}_3
=(\lambda+2\mu) |\partial_3 u_3|^2+\lambda (\partial_1u_1+\partial_2u_2)\,\partial_3\overline{u}_3.
\end{align*}

Hence,
\begin{align}
&\int_{S_H}2\Re\Big\{\sum_{j=1}^3\mathcal{E}(u,(x_3-m)e_j)\partial_3\bar{u}_j\Big\}{\rm d}x\notag\\
&= 2(\lambda + 2\mu)\|\partial_3u_3\|^2_{L^2(S_H)^3} + 4\mu(\|\partial_3u_1\|^2_{L^2(S_H)^3} + \|\partial_3u_2\|^2_{L^2(S_H)^3})\notag\\
&\quad + 2\lambda \Big(\Re\int_{S_H}\partial_1u_1\partial_3\bar{u}_3{\rm d}x + \Re\int_{S_H}\partial_2u_2\partial_3\bar{u}_3{\rm d}x\Big)\notag\\
&\quad -2\mu {\color{black}{\Re}} \left\{ \int_{S_H}(\partial_3u_1 - \partial_1u_3)\partial_3\bar{u}_1 - (\partial_2u_3 - \partial_3u_2)\partial_3\bar{u}_2{\rm d}x \right\}.
\end{align}
Choosing $C>0$ to be sufficiently large, we get
\begin{align}\label{5.43}
&\int_{S_H}2\Re\Big\{\sum_{j=1}^3\mathcal{E}(u,(x_3-m)e_j)\partial_3\bar{u}_j\Big\}{\rm d}x + C\|\nabla\cdot u\|^2_{L^2(S_H)} + C\|\nabla\times u\|^2_{L^2(S_H)^3}\notag\\
&=I_1+I_2+I_3+ C\|\partial_1u_2 - \partial_2u_1\|^2_{L^2(S_H)},
\end{align}
where
\begin{align}
I_1&:= [C + 2(\lambda + 2\mu)]\|\partial_3u_3\|^2_{L^2(S_H)^3} + C\|\partial_1u_1\|^2_{L^2(S_H)^3} + C\|\partial_2u_2\|^2_{L^2(S_H)^3}\notag\\
&\quad + (C+2\lambda)\Big(\Re\int_{S_H}\partial_1u_1\partial_3\bar{u}_3{\rm d}x + \Re\int_{S_H}\partial_2u_2\partial_3\bar{u}_3{\rm d}x\Big)
+C\Re\int_{S_H}\partial_1u_1\partial_2\bar{u}_2{\rm d}x,\notag\\
&=\int_{S_H}A [\partial_1u_1, \partial_2u_2, \partial_3u_3]^\top \cdot \,
[\partial_1\bar{u}_1, \partial_2\bar{u}_2, \partial_3\bar{u}_3]^\top{\rm d}x,\notag \\
&A:=\begin{pmatrix}
C & C/2 & \lambda+C/2\\
C/2 & C & \lambda+C/2 \\
\lambda+C/2 & \lambda+C/2 & C+ 2(\lambda+2\mu)
\end{pmatrix},\notag \\
I_2&:= 4\mu \|\partial_3u_1\|^2_{L^2(S_H)^3}
+ C\|\partial_3u_1 - \partial_1u_3\|^2_{L^2(S_H)} -2\mu\,{\color{black}{\Re}}  \,\int_{S_H}(\partial_3u_1 - \partial_1u_3)\partial_3\bar{u}_1{\rm d}x,\notag\\
I_3&:= 4\mu \|\partial_3u_2\|^2_{L^2(S_H)^3})+ C\|\partial_2u_3 - \partial_3u_2\|^2_{L^2(S_H)} +2\mu\,{\color{black}{\Re}}\,  \int_{S_H}(\partial_2u_3 - \partial_3u_2)\partial_3\bar{u}_2{\rm d}x.\notag
\end{align}
Direct calculations show that $\mbox{Det}(A)\sim C^2/8$ as $C\rightarrow\infty$. Hence the matrix $A\in \mathbb{R}^{3\times 3}$ must be strictly positive for sufficiently large $C>0$. This gives
\begin{align}
I_1\geq C_0\,(\|\partial_1 u_1\|^2_{L^2(S_H)}+
\|\partial_2 u_2\|^2_{L^2(S_H)}+\|\partial_3 u_3\|^2_{L^2(S_H)}),
\end{align}
where the constant $C_0>0$ only depends on $\lambda$ and $\mu.$ By arguing in the same manner one has for $C>\mu^2/4$ that
\begin{align}
I_2\geq C_0\,(\|\partial_3 u_1\|^2_{L^2(S_H)}+  \|\partial_3 u_1-\partial_1u_3\|^2_{L^2(S_H)}), \,\\ \label{4.36}
I_3\geq C_0\,(\|\partial_3 u_2\|^2_{L^2(S_H)}+  \|\partial_3 u_2-\partial_2u_3\|^2_{L^2(S_H)}).
\end{align}
Hence, it follows from \eqref{5.43}-\eqref{4.36} that
\begin{align}\label{5.45}
&\int_{S_H}2\Re\Big\{\sum_{j=1}^3\mathcal{E}(u,(x_3-m)e_j)\partial_3\bar{u}_j\Big\}{\rm d}x + C\|\nabla\cdot u\|^2_{L^2(S_H)} + C\|\nabla\times u\|^2_{L^2(S_H)^3}\notag \\
&\geq C_0\,(\|\partial_1 u_1\|^2_{L^2(S_H)}+
\|\partial_2 u_2\|^2_{L^2(S_H)}+\|\partial_3 u_3\|^2_{L^2(S_H)}+ \|\partial_1u_2 - \partial_2u_1\|^2_{L^2(S_H)}) \notag\\
&\quad +C_0\,(\|\partial_3 u_1\|^2_{L^2(S_H)}+  \|\partial_1u_3\|^2_{L^2(S_H)}+\|\partial_3 u_2\|^2_{L^2(S_H)}+\|\partial_2 u_3\|^2_{L^2(S_H)})   ,
\end{align}
provided $C>0$ is sufficiently large.
Combining \eqref{5.33}, \eqref{5.42} and \eqref{5.45} and using Young's inequality gives
\begin{align}\label{5.44}
\mbox{Right hand side of \eqref{5.45}}\leq (C_4(\omega,h)^2+C_5(\omega,h)^2+C_6(\omega,h,L)^2)\|g\|^2_{V_h}.
\end{align}
However, we still need to
 estimate
$\|\partial_1u_2\|^2_{L^2(S_H)}$ and $\|\partial_2u_1\|^2_{L^2(S_H)}$.
Since
$\|\partial_3u\|^2_{L^2(S_H)^3}$ can also be bounded by the right hand side of \eqref{5.44}, we have (see \cite[Lemma 3.4]{CM})
\begin{align}\label{5.46}
\|u\|^2_{L^2(S_H)}\leq C_0\,\|\partial_3u\|^2_{L^2(S_H)^3}\leq
 (C_4(\omega,h)^2+C_5(\omega,h)^2+C_6(\omega,h,L)^2)\|g\|^2_{V_h}.
\end{align}
Now, using \eqref{5.37}, \eqref{5.40} and \eqref{5.46} we arrive at
\begin{align}\label{5.47}
\mathcal{E}(u, \bar{u})\leq (C_4(\omega,h)^2+C_5(\omega,h)^2+C_6(\omega,h,L)^2)\|g\|^2_{V_h}.
\end{align}
Recalling the expression  of $\mathcal{E}$, we find
\begin{align*}
2\mu(\partial_1 u_2+\partial_2 u_1)=\mathcal{E}(u,u)-\lambda\,\nabla\cdot u-\mu \,\nabla\times u-2\mu(\partial_1 u_3+\partial_1 u_1+\partial_2u_2+\partial_2u_3+\sum_{j=1}^3 \partial_3u_j).
\end{align*}
It follows from \eqref{5.47}, \eqref{5.33} and \eqref{5.44} that each term on the right hand side of the previous identity can be bounded by the right hand side of \eqref{5.47}, leading to the same upper bound for $||\partial_1 u_2+\partial_2 u_1||_{L^2(S_H)}$. Finally, recalling the upper bound for the difference $||\partial_1 u_2-\partial_2 u_1||_{L^2(S_H)}$ (see \eqref{5.45}) we obtain the estimates for  $||\partial_1 u_2||_{L^2(S_H)}^2$ ,$||\partial_2 u_1||^2_{L^2(S_H)}$ and thus also for $||\nabla u||$. Using the $L^2$-estimate for $u$ (see \eqref{5.46}) we obtain
\begin{align*}
||u||^2_{V_h}\leq
 (C_4(\omega,h)^2+C_5(\omega,h)^2+C_6(\omega,h,L)^2)\|g\|^2_{V_h}.
\end{align*}

Now the {\it a priori} bound for $f$ being smooth has been proved. It can be extended to the case of a general Lipschitz function by the method of approximation
in \cite{EH-SIMA}. This completes the proof.

\end{proof}

\section{Well-posedness for random rough surfaces}\label{random}

In this section, we investigate the well-posedness of elastic scattering by a random rough surface. Let  $(\Omega,\mathcal{A},\mathbb{P})$ be a complete probability space. Denote by $S(\eta)$ a random surface
\[
\Gamma(\eta):=\{x \in \mathbb{R}^3 :x_3=f(\eta;x_1,{x}'),\eta \in \Omega, {x}' \in \mathbb{R}^2\}.
\] Similarly, $D(\eta)$ and $S_h(\eta)$ represent the random counterparts of $D$ and $S_h$, respectively. Assume $f(\eta;x')$ is a Lipschitz continuous function with Lipschitz constant $L(\eta)$ for all $\eta \in \Omega$ and it also satisfies $m<f(\eta; x')<M$.
The random source $g(\eta)$ is assumed to satisfy $g(\eta) \in L^2(D(\eta))^3$ with its support in $S_h(\eta)$. Similarly as the deterministic case, we can give the following random boundary value problem.\begin{align*}
	\begin{array}{rll}
		\Delta^*u(\eta;\cdot)+\omega^2u(\eta;\cdot)=g(\eta;\cdot)&{\rm in}&S_h(\eta), \\
		u(\eta;\cdot)=0&{\rm on}&\Gamma(\eta), \\
		Tu(\eta;\cdot)=\mathcal{T}u(\eta;\cdot)&{\rm on}&\Gamma_h.
	\end{array}
\end{align*}
For simplicity, let $V_h(\eta)=V_h(S_h(\eta))$.
Define a sesquilinear form $\tilde{B}_\eta$ on $V_h(\eta)\times V_h(\eta)$ by \begin{equation}  \label{eq6.1}
	\tilde{B}_\eta(u,v)=
	\int_{S_h(\eta)} \mathcal{E}(u,\bar{v})-\omega^2 u\cdot\bar{v}\,\mathrm{d}x-\int_{\Gamma_h}\mathcal{T}u\cdot\bar{v}\,\mathrm{d}s,
\end{equation}
and an antilinear functional $\tilde{G}_\eta$ on $V_h(\eta)$ by
\begin{equation}  \label{eq6.2}
	\tilde{G}_\eta(v):=-\int_{S_h(\eta)}g(\eta)\cdot\bar{v}\,\text{d}x.
\end{equation}
To define the stochastic variation problem directly is not suitable since  $V_h(\eta)$ is dependent on $\eta$. We take a variable transform to give a new sesquilinear form defined on $V_h \times V_h$. Let $f_0=f(\eta_0)$ and $g_0=g(\eta_0)$ for some fixed $\eta_0 \in \Omega$ and write $D=D(\eta_0)$, $S_h=S_h(\eta_0)$  and $ V_h=V_h(\eta_0)$ for convenience. In addition, we assume that $g(\eta) \in H^1(D(\eta))^3$
and \[
\|f(\eta)-f_0\|_{1,\infty} \le M_0, \quad \forall \eta \in \Omega,
\] with some constant $M_0 >0$. Moreover, the truncated height $h$ is chosen such that\begin{equation}\label{eq6.3}
	(M-m)/\gamma<1,
\end{equation} where $\gamma=h-\sup\limits_{x'}f_0(x')$. This condition ensures the invertiblity of the variable transform $\mathcal{H}$ which will be introduced later. Since $\Gamma_h$ is artificial, choosing sufficiently large $h$ will be enough.

Denote by $Lip (\mathbb{R}^2)$ the set including all Lipschitz continuous functions on $\mathbb{R}^2$. Then define a product topology space \[
\mathcal{C}= \mathcal{C}_1 \times \mathcal{C}_2,
\]where\[
\mathcal{C}_1:=\{v \in Lip(\mathbb{R}^2) :m<v<M, \|v-f_0\|_{1,\infty}\le M_0\},
\] with constant $M_0>0$
and\[
\mathcal{C}_2:=H^1_0(S_h)^3.
\]The topology of $ \mathcal{C}_1$ and $\mathcal{C}_2$ are respectively given by the norms $\|\cdot\|_{1,\infty}$ and $\|\cdot\|_{H^1(S_h)^3}$.

Consider the transform $\mathcal{H}$: $S_h \to S_h(\eta)$ defined by \begin{equation*}
	\mathcal{H}(y)=y+\alpha(y_3-f_0(y'))(f(\eta; y')-f_0(y'))e_3,\quad y \in D_h,
\end{equation*}
where $e_3$ is the unit vector in $x_3$ direction and $\alpha(x)$ is a cutoff function which satisfies\[
\alpha(x)=\left\{\begin{array}{cc}
	0, &  x<\delta, \\
	1, &  x>\gamma,
\end{array}\right.
\]with sufficiently small $\delta$.
It is also required to satisfy\begin{equation}
	\label{eq6.4}|\alpha'|<1/(\gamma-2\delta).
\end{equation}
The Jacobi matrix of $\mathcal{H}$ is\begin{equation*}
	\mathcal{J}_\mathcal{H}=I_3+\left(\begin{array}{ccc}
		0 & 0 & 0 \\
		0 & 0 & 0 \\
		J_1 & J_2 & J_3
	\end{array}\right),
\end{equation*}
where \begin{align*}
	J_i=\alpha(y_3-f_0(y'))(\partial_i f(\eta;y')-\partial_i f_0(y'))-\alpha'(y_3-f_0(y'))\partial_i f_0(y_1)(f(\eta;y')-f_0(y')),\quad i=1,2\end{align*} and
\begin{align*}
	J_3=\alpha'(y_3-f_0(y'))(f(\eta;y')-f_0(y')).
\end{align*}
Since matrix $\mathcal{J}_\mathcal{H}$ is required to be non-singular so that $\mathcal{H}$ is invertible, according to \eqref{eq6.4},  we obtain \[
|J_3|<\frac{M-m}{\gamma-2\delta}.
\] Hence, by \eqref{eq6.3}, we can choose $\delta$ sufficiently small such that \begin{equation*}
	|J_3|<\frac{M-m}{\gamma-2\delta}<1,
\end{equation*} which implies that $\mathcal{H}$ is invertible.
It is easy to verify $\mathcal{H}(\Gamma_h)=\Gamma_h$. Set \[
A=(\alpha_1, \alpha_2, \alpha_3), \, B^\top =(\beta_1, \beta_2, \beta_3) \in \mathbb{C}^{3\times3},
\] then denote \[
A:B=\text{tr}(B^\top A)
\] and \[
A \otimes B=\left(\begin{array}{c}
	\alpha_2 \cdot \beta_3-\alpha_3 \cdot \beta_2 \\
	\alpha_3 \cdot \beta_1-\alpha_1\cdot\beta_3 \\
	\alpha_1 \cdot \beta_2 -\alpha_2 \cdot \beta_1
\end{array}\right).
\]
For $u,v \in V_h(\eta)$, taking $x=\mathcal{H}(y)$ in \eqref{eq6.1} yields
\begin{align*}
	\tilde{B}_\eta(u,v)=&2\mu\int_{S_h} \sum_{j=1}^3\nabla\tilde{u}_j\mathcal{J}_{\mathcal{H}^{-1}}\mathcal{J}_{\mathcal{H}^{-1}}^\top\nabla \bar{\tilde{v}}_j\det{\mathcal{J}_\mathcal{H}}\,\text{d}y \\
	&+\lambda\int_{S_h} (\nabla\tilde{u}:\mathcal{J}_{\mathcal{H}^{-1}}^\top)(\nabla\bar{\tilde{v}}:\mathcal{J}_{\mathcal{H}^{-1}}^\top)\det{\mathcal{J}_\mathcal{H}}\,\text{d}y \\& -\mu \int_{S_h} (J_{\mathcal{H}^{-1}} \otimes \nabla \tilde{u})(J_{\mathcal{H}^{-1}} \otimes \nabla\bar{\tilde{v}})\det{\mathcal{J}_{\mathcal{H}}}\,\mathrm{d}y \\
	&-\omega^2\int_{S_h}\tilde{u}\cdot\bar{\tilde{v}}\det{\mathcal{J}_{\mathcal{H}}}\,\text{d}y
	-\int_{\Gamma_h}\mathcal{T}\tilde{u}\cdot\bar{\tilde{v}} \,\text{d}s(y),
\end{align*}
where $\tilde{u}=u\circ\mathcal{H}$, $\tilde{v}=v\circ\mathcal{H}$.
Similarly, for $v \in V_h(\eta)$, let $x=\mathcal{H}(y)$ in \eqref{eq6.2}, \begin{equation*}
	\tilde{G}_\eta(v)=-\int_{D_h}\tilde{g}(\eta)\cdot\bar{\tilde{v}}\det{\mathcal{J}_\mathcal{H}}\,\text{d}x.
\end{equation*}
Recall that we require $g(\eta) \in H^1(D(\eta))^3$ and the support of $g(\eta)$ is in $S_h(\eta)$, we have $\tilde{g}(\eta) \in H^1_0(S_h)^3$ for all $\eta$. So we can define the input map $c$ : $\Omega\to\mathcal{C}$ by \[
c(\eta):=(f(\eta),\tilde{g}(\eta)).
\]
Note that $\tilde{u},\tilde{v} \in V_h$. Thus we can define a continuous sesquilinear form $B_{c(\eta)}(u,v)$ on $V_h \times V_h$ by \begin{align}\label{eq6.6}
	B_{c(\eta)}(u,v):=&2\mu\int_{S_h} \sum_{j=1}^3\nabla{u}_j\mathcal{J}_{\mathcal{H}^{-1}}\mathcal{J}_{\mathcal{H}^{-1}}^\top\nabla\bar{{v}}_j\det{\mathcal{J}_\mathcal{H}}\,\text{d}y \notag \\
	&+\lambda\int_{S_h} (\nabla{u}:\mathcal{J}_{\mathcal{H}^{-1}}^\top)(\nabla\bar{{v}}:\mathcal{J}_{\mathcal{H}^{-1}}^\top)\det{\mathcal{J}_\mathcal{H}}\,\text{d}y \notag \\& -\mu \int_{S_h} (J_{\mathcal{H}^{-1}} \otimes \nabla {u})(J_{\mathcal{H}^{-1}} \otimes \nabla\bar{{v}})\det{\mathcal{J}_{\mathcal{H}}}\,\mathrm{d}y\notag\\
	&-\omega^2\int_{S_h}{u}\cdot\bar{{v}}\det{\mathcal{J}_\mathcal{H}}\,\text{d}y
	-\int_{\Gamma_h}\mathcal{T}{u}\cdot\bar{{v}} \,\text{d}s(y).
\end{align}
It is easy to see\begin{equation*}
	\tilde{B}_\eta(u,v)=B_{c(\eta)}(\tilde{u},\tilde{v}).
\end{equation*}
Similarly we can define an antilinear functional $G_{c(\eta)}$ on $V_h$ by \begin{equation}\label{eq6.7}
	G_{c(\eta)}(v):=-\int_{S_h}\tilde{g}(\eta)\cdot\bar{{v}}\det{\mathcal{J}_\mathcal{H}}\,\text{d}x.
\end{equation}
Obviously, there holds the identity \begin{equation*}
	G_{c(\eta)}(\tilde{v})=\tilde{G}_\eta(v).
\end{equation*} 

Then the sesquilinear form $\tilde{\mathcal{B}}$ on $L^2(\Omega; V_h)\times L^2(\Omega; V_h)$ can be defined by\begin{equation*}
	\mathcal{B}(u,v):=\int_\Omega B_{c(\eta)}(u,v)\,\text{d}\mathbb{P}(\eta)
\end{equation*}
and the antilinear functional $\mathcal{G}$ is defined on $L^2(\Omega;V_h)$ by\begin{equation*}
	\mathcal{G}(v):=\int_\Omega G_{c(\eta)}(v)\,\text{d}\mathbb{P}(\eta).
\end{equation*}
For convenience, we regard the sesquilinear form $B_{c(\eta)}$ : $V_h \times V_h \to \mathbb{C}$ as the same operator in $B(V_h, V_h^*)$ generated by it. Here $V_h^*$ is the dual space of $V_h$ and $ B(X,Y)$ denote the space including all bounded linear operators $X \to Y$. Similarly to \eqref{eq6.6}-\eqref{eq6.7}, we can define the sesquilinear form $B_{(\phi,\psi)}$ and the antilinear functional $G_{(\phi,\psi)}$ for all $(\phi,\psi) \in \mathcal{C}$. Then we can define the map $\mathscr{B}$: $\mathcal{C} \to B(V_h,V_h^*) $ by\[
\mathscr{B}((\phi,\psi)):=B_{(\phi,\psi)}
\]
and the map $\mathscr{G}$ : $\mathcal{C} \to  V_h^* $ by
\[
\mathscr{G}((\phi,\psi)):=G_{(\phi,\psi)}.
\]
Now we can define the stochastic variation problem as follows.

{\it Variational Problem II:}
find $u \in L^2(\Omega;V_h)$ such that \begin{equation}\label{eq6.8}
	\mathcal{B}(u,v)=\mathcal{G}(v), \quad \forall v\in L^2(\Omega;V_h).
\end{equation}

We will consider the well-posedness of the stochastic variation problem \eqref{eq6.8}.
Firstly we show both the sesquilinear form $\mathcal{B}$ and the antilinear functional $\mathcal{G}$ are well-defined which is based on measurability and $\mathbb{P}$-essentially separability of $c$.
For measurability and $\mathbb{P}$-essentially separability of $c$, the following condition is necessary.
\begin{Co}\label{condition1}
	The map $c_1$: $\Omega \to \mathcal{C}_1$ defined by \[
	c_1(\eta)=f(\eta)
	\] satisfies $c_1 \in L^2(\Omega; \mathcal{C}_1)$ and the map $c_2$: $\Omega \to \mathcal{C}_2$ defined by \[
	c_2(\eta)=\tilde{g}(\eta)
	\] satisfies $c_2 \in L^2(\Omega; \mathcal{C}_2)$.
\end{Co}
It implies the following lemma (see Lemma 4.1 in \cite{WLX}).
\begin{lemm}
	Under Condition \ref{condition1}, the map $c$ is measurable and $\mathbb{P}$-essentially separable.
\end{lemm}
Then we can prove that the sesquilinear form $\mathcal{B}$ is well-defined by the continuity of $\mathscr{B}$ and the regularity of $\mathscr{B}\circ c$.
\begin{lemm} \label{lemma5.2}
	\begin{itemize}
		\item[(i)]  The map $\mathscr{B}$: $\mathcal{C} \to B(V_h,V_h^*) $ is continuous.
		\item [(ii)] The map $\mathscr{B}\circ c \in L^\infty(\Omega; B(V_h,V_h^*)).$
		\item[(iii)] The sesquilinear form $\mathcal{B}$ is well-defined on $L^2(\Omega ; V_h) \times L^2(\Omega ; V_h). $
	\end{itemize}
\end{lemm}
\begin{proof} We only prove (i), since (ii),(iii) can be verified similarly as the two-dimensions case in \cite{WLX}.
	For convenience, we only prove the continuity at the point $(f_0,g_0)\in \mathcal{C}$ since for other points the proof is similar. Consider the sequence $\{(f_m, g_m)\} \subset \mathcal{C}$ such that $(f_m, g_m) \to (f_0,g_0)$ in $\mathcal{C}$ when $m \to \infty$.
	Denote the transform by
	\begin{equation*}
		\mathcal{H}_m(y)=y+\alpha(y_3-f_0(y'))(f_m(y')-f_0(y'))e_3,\quad y \in D_h.
	\end{equation*}
	For any $u,v \in V_h$, \begin{align*}
		B_{(f_m, g_m)}(u,v)&-B(u,v)=
		2\mu\int_{S_h} \sum_{j=1}^2\nabla u_j(I_3-\mathcal{J}_{\mathcal{H}^{-1}_m}\mathcal{J}_{\mathcal{H}^{-1}_m}^\top\det{\mathcal{J}_{\mathcal{H}_m}})\nabla \bar{v}_j\,\text{d}x \\
		&+\lambda\int_{S_h} (\nabla\cdot u)(\nabla\cdot \bar{v})-(\nabla\tilde{u}:\mathcal{J}_{\mathcal{H}^{-1}_m})(\nabla\bar{\tilde{v}}:\mathcal{J}_{\mathcal{H}^{-1}_m}^\top)\det{\mathcal{J}_{\mathcal{H}_m}}\,\text{d}x \\ &-\mu \int_{S_h}(J_{\mathcal{H}_m^{-1}} \otimes \nabla \tilde{u})(J_{\mathcal{H}_m^{-1}} \otimes \nabla \bar{\tilde{v}})\det{J_{ \mathcal{H}_m}}-(\nabla  \times u) \cdot (\nabla \times \bar{v}) \,\mathrm{d}x \\
		&-\omega^2\int_{S_h}u\cdot\bar{v}(\det{\mathcal{J}_{\mathcal{H}_m}}-1)\,\text{d}x.
	\end{align*}
	By direct calculations, we have\begin{equation*}
		\det{\mathcal{J}_{\mathcal{H}_m}}=1+O(\|f_m-f_0\|_{1,\infty}), \quad \mathcal{J}_{\mathcal{H}^{-1}_m}=I_3+O(\|f_m-f_0\|_{1,\infty}),
	\end{equation*} which imply that\begin{equation*}
		|B_{(f_m,g_m)}(u,v)-B(u,v)|\le C\|u\|_{H^1(D_h)^2}\|v\|_{H^1(S_h)^3}\|f_m-f_0\|_{1,\infty}.
	\end{equation*}
	It turns out when $m \to \infty$, \begin{equation*}
		\|B_{(f_m,g_m)}-B\|_{B(V_h,V_h^*)}\le C \|f_m-f_0\|_{1,\infty} \to 0.
	\end{equation*} This completes the proof.
\end{proof}
Next we give a similar lemma for the antilinear functional $\mathcal{G}$.
\begin{lemm} \label{lemma5.3}
	\begin{itemize}
		\item[(i)] 	The map $\mathscr{G}$: $\mathcal{C}\to V_h^*$ is continuous.
		\item[(ii)] The map $\mathscr{G}\circ c \in L^2(\Omega; V_h^*).$
		\item[(iii)]  The antilinear functional $\mathcal{G}$ is well-defined on $L^2(\Omega ; V_h). $
	\end{itemize}
\end{lemm}
The proof is similar to Lemma 4.3 in \cite{WLX}.
For any given sampling $\eta$, we consider the following deterministic \it Variational Problem III.

Find $u(\eta) \in V_h$ such that
\begin{equation}\label{eq6.9}
	B_{c(\eta)}(u(\eta),v)=G_{c(\eta)}(v),\quad \forall v \in V_h.\end{equation}

The existence and uniqueness of solutions of the problem \eqref{eq6.9} has been given in Theorem \ref{thm}. The \textit{a priori} bound in Lemma \ref{lemma7} can also be used for \eqref{eq6.9}. Notice that for any $\eta$ we have the upper bound \[
L(\eta) \le L+M_0.
\] 
\begin{lemma}\label{lemma_random}
	For any given $\eta$, the variational problem \eqref{eq6.9} admits a unique solution $u(\eta) \in V_h$.
Moreover, the \textit{a priori} bound\[\|u^*(\eta)\|_{H^1(
		S_h(\eta))^3}
	\le (h-m+2)(C_4(\omega, h )+C_5(\omega, h )+C_6(\omega, h,L_0) ) \|g(\eta)\|_{H^1(S_h(\eta))^3}
	\] holds for $u^*(\eta)=u(\eta)\circ \mathcal{H}^{-1}$ with $L_0=M_0+L$.
\end{lemma}
\begin{proof}
If $u(\eta)$ is a solution to \textcolor{black}{\it Variational Problem III} \eqref{eq6.9}, then $u^*(\eta)=u(\eta)\circ \mathcal{H}^{-1}$ is solution to \textcolor{black}{\it Variational Problem I} \eqref{3.16} corresponding to $f(\eta)$ and $g(\eta)$. Conversely, if $u(\eta)$ is solution to \textcolor{black}{\it Variational Problem I} \eqref{3.16} corresponding to $f(\eta)$ and $g(\eta)$, then $\tilde{u}(\eta)=u(\eta)\circ \mathcal{H}$ is solution to \textcolor{black}{\it Variational Problem III} \eqref{eq6.9}. So Theorem \ref{thm} implies existence and uniqueness of solutions to the variation problem (\ref{eq6.8}), and Theorem \ref{lemma5} implies the \textit{a priori} bound.
\end{proof}
Lemma \ref{lemma_random} shows the existence of a solution $u(\eta)$ to \eqref{eq6.9} for given $\eta$. In fact, the following lemma shows $u(\eta) \in L^2(\Omega; V_h)$.
\begin{lemm}\label{lemma5.5}
	For the solution $u(\eta)$ to \textcolor{black}{\it Variational Problem III} \eqref{eq6.9}, we have $u(\eta) \in L^2(\Omega; V_h)$.
\end{lemm} The proof is omitted here since it is similar to the two-dimensions case in Lemma 4.4 in \cite{WLX}. Based on Lemmas \ref{lemma5.2} - \ref{lemma5.5},  we can conclude the well-posedness of \eqref{eq6.8} in the framework of \cite{PS, BLX, WLX} and extend the  \textit{a priori} bound to random case as follows.
\begin{theo} \label{thm_random}
	\begin{itemize}
		\item [(i)] The Variational Problem  II \eqref{eq6.8} admits a unique solution $u\in L^2(\Omega, V_h)$.
		\item [(ii)] Let $u\in V_h(\eta)$ be a solution to the Variation Problem I \eqref{3.16} corresponding to $f(\eta)$ and $g(\eta)$ with $\eta \in \Omega$, and let $\tilde{u}(\eta)\in L^2(\Omega;V_h)$ be the solution to the Variational Problem II \eqref{eq6.8}. Then $u$ and $\tilde{u}$  satisfy respectively the bound \begin{align*}
			&\int_\Omega\|u\|^2_{H^1(S_h(\eta))^3}\mathrm{d}\,\mathbb{P}\\ &\le (h-m+2)^2(C_4(\omega, h )+C_5(\omega, h )+C_6(\omega, h,L_0) )^2\int_{\Omega}\|g\|^2_{H^1(S_h(\eta))^3}\mathrm{d}\,\mathbb{P},
		\end{align*} and
		\begin{align*} &\int_\Omega\|\tilde{u}\|^2_{H^1(S_h)^3}\mathrm{d}\,\mathbb{P}\\ &\le (h-m+2)^2(C_4(\omega, h )+C_5(\omega, h )+C_6(\omega, h,L_0) )^2\int_{\Omega}\|\tilde{g}\|^2_{H^1(S_h)^3}\mathrm{d}\,\mathbb{P}.
		\end{align*}
	\end{itemize}
\end{theo}

\section{Conclusion} \label{conclusion}

We establishes the well-posedness of the time-harmonic elastic scattering from general unbounded rough surfaces in three dimensions at an arbitrary frequency.
{\color{black}{\it A priori} bounds which are explicit dependent on the frequency and on the geometry of the rough surface are derived both for deterministic and random cases.} A possible continuation of this work
is to study the elastic scattering by incident plane waves, spherical or cylindrical waves. We hope to report the progress on these results in subsequent publications.

%
%

\end{document}